\documentclass[11pt]{amsart}

\usepackage{enumitem}
\usepackage{verbatim}
\usepackage{hyperref}
\usepackage{graphicx}
\usepackage{xcolor}

\usepackage{amssymb}
\usepackage{amsmath}
\usepackage{amsthm}

\newcommand{\eps}{\varepsilon}

\newcommand{\Z}{\mathbb{Z}}
\newcommand{\R}{\mathbb{R}}

\newcommand{\inj}{\mathrm{inj}}

\newcommand{\grad}{\mathrm{grad}}

\newcommand{\sqbd}{\partial^2 X}

  \newcommand{\Pa}{\mathbf{P}}
\newcommand{\Fu}{\mathbf{F}}

\newtheorem{MainThm}{Theorem}
\newtheorem{theorem}{Theorem}[section]
\newtheorem{lemma}[theorem]{Lemma}
\newtheorem{proposition}[theorem]{Proposition}
\newtheorem{corollary}[theorem]{Corollary}
\newtheorem{definition}[theorem]{Definition}

\newtheorem*{thma*}{Theorem}

\theoremstyle{remark}
\newtheorem{remark}[theorem]{Remark}

\DeclareMathOperator{\arcosh}{\mathrm{arcosh}}

\DeclareMathOperator{\id}{Id}
\DeclareMathOperator{\pr}{pr}
\numberwithin{equation}{section}
\numberwithin{figure}{section}

\begin{document}
	
\UseRawInputEncoding

\title[]{A quantitative closing Lemma and Partner orbits  on Riemannian manifolds with negative curvature}
\author{Michela Egidi}
\address{Institute of Mathematics, University of Rostock, 18051 Rostock, Germany}
\author{Gerhard Knieper}
\address{Dept.\ of Mathematics, Ruhr University Bochum, 44780 Bochum, Germany}
\email{michela.egidi@uni-rostock.de}
\email{gerhard.knieper@rub.de}
\date{\today}

\begin{abstract}
In this paper we prove a quantitative closing Lemma for manifolds of negative sectional curvature.
As an application we study partner and pseudo-partner orbits for self-crossing closed geodesic.
\end{abstract}

\thanks{The authors are partially supported by the German Research Foundation (DFG),
CRC TRR 191, \textit{Symplectic structures in geometry, algebra and dynamics.}}
\maketitle

\section{Introduction}\label{sec:intro}
In this paper we provide a quantitative closing Lemma for closed Riemannian manifolds of negative sectional curvature.
More precisely, let $(M,g)$ be a closed Riemannian manifold of dimension $d\geq2$ whose sectional curvature $K$ is bounded
by $- \kappa_2^2 \le K \le - \kappa_1^2$ for some $0 < \kappa_1 \le  \kappa_2 < \infty$
and let $\phi^t: SM \to SM$ be the geodesic flow on the unit tangent bundle $SM$ of $M$. The metric $d$ induced by $g$ on $M$ yields a metric $d_1$  on  $SM$  given by
$$
d_1(v,w)=\max_{t\in[-1,1]} d(c_v(t), c_w(t))
$$
where  $c_v,c_w$ are geodesics in $M$ with initial vector $v$ and $w$, respectively.
Under these assumptions we obtain the following quantitative version of the closing Lemma (see  Theorem \ref{thm:closing} in Section \ref{sec:closing}).

\begin{MainThm}\label{thm:closing-i}
There exist $\delta_0=\delta_0(\kappa_1,\kappa_2)\in\big(0,\frac{1}{2}\big)$ and $t_0=t_0(\kappa_1,\kappa_2)>1$
such that for all $\delta\leq \delta_0$, all $T\geq t_0$ and all $w\in SM$ with $d_1(w,\phi^T(w))\leq \delta$
there exist $u\in SM$, $T'>0$, with $\phi^{T'}(u)=u$ and
\[
\vert T-T' \vert \leq 2 C \delta,
\]
where the constant $C$ is given by
$$
C= \frac{4\pi}{\kappa_1}\Big(\frac{2\kappa_2}{\kappa_1}+3\Big).
$$
Furthermore,
\[
d_1(\phi^s(w), \phi^s(u))\leq (5C +1)\delta \qquad\forall\, s\in[0,T].
\]
\end{MainThm}

A first consequence of the closing Lemma is the existence of so called partner orbits of closed geodesics with small self-crossing angle.
A closed geodesic $c_w:[0,T] \to M$  with $c_w(T_1) = c_w(0) $ for $0 <T_1 <T$ has a self-crossing at $p = c_w(0)$ at time $T_1$ and we denote by
$\sphericalangle_p (w, -\dot c_w(T_1))$ its crossing angle.

Partner orbits have been studied by physicist to
explain universal behaviour of the spectrum of the associated quantized systems (see \cite{ABHHKM08} for more details).
These partner orbits are new closed geodesics with slightly smaller length whose image lies in a very small neighbourhood  of the original closed geodesic.
For surfaces they have been first studied by Sieber and Richter \cite{SR01, mS02} and such pairs are sometimes referred to as Sieber-Richter pairs.
A rigorous mathematical treatment has been provided for closed surfaces of constant negative curvature by Huynh and Kunze in \cite{HK15} and  Huynh in \cite{H16}. While in \cite{HK15} and \cite{H16}
the authors used an algebraic approach, we will use more geometric techniques which work for Riemannian manifolds with variable negative curvature
and arbitrary dimension (see Theorem \ref{thm:partner-general} in Section \ref{sec:partner-orbits-general}).

\begin{figure}[htb]\label{fig:self-intersection}
\begingroup%
  \makeatletter%
  \providecommand\color[2][]{%
    \errmessage{(Inkscape) Color is used for the text in Inkscape, but the package 'color.sty' is not loaded}%
    \renewcommand\color[2][]{}%
  }%
  \providecommand\transparent[1]{%
    \errmessage{(Inkscape) Transparency is used (non-zero) for the text in Inkscape, but the package 'transparent.sty' is not loaded}%
    \renewcommand\transparent[1]{}%
  }%
  \providecommand\rotatebox[2]{#2}%
  \newcommand*\fsize{\dimexpr\f@size pt\relax}%
  \newcommand*\lineheight[1]{\fontsize{\fsize}{#1\fsize}\selectfont}%
  \ifx\svgwidth\undefined%
    \setlength{\unitlength}{338.7478329bp}%
    \ifx\svgscale\undefined%
      \relax%
    \else%
      \setlength{\unitlength}{\unitlength * \real{\svgscale}}%
    \fi%
  \else%
    \setlength{\unitlength}{\svgwidth}%
  \fi%
  \global\let\svgwidth\undefined%
  \global\let\svgscale\undefined%
  \makeatother%
  \begin{picture}(1,0.28502031)%
    \lineheight{1}%
    \setlength\tabcolsep{0pt}%
    \put(0,0){\includegraphics[width=\unitlength,page=1]{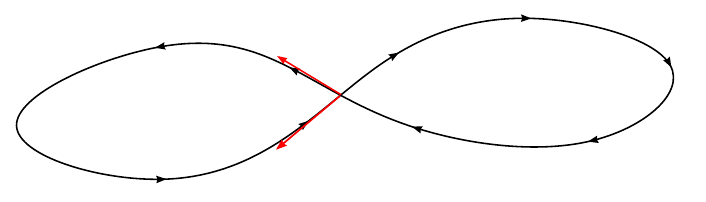}}%
    \put(0.4421747,0.17465525){\makebox(0,0)[lt]{\lineheight{1.25}\smash{\begin{tabular}[t]{l}$w$\end{tabular}}}}%
    \put(0.28310227,0.08916732){\makebox(0,0)[lt]{\lineheight{1.25}\smash{\begin{tabular}[t]{l}$-\dot c_w(T_1)$\end{tabular}}}}%
    \put(0,0){\includegraphics[width=\unitlength,page=2]{self-intersection.pdf}}%
    \put(0.39456283,0.13286327){\makebox(0,0)[lt]{\lineheight{1.25}\smash{\begin{tabular}[t]{l}$\varepsilon$\end{tabular}}}}%
    \put(0,0){\includegraphics[width=\unitlength,page=3]{self-intersection.pdf}}%
    \put(0.06234864,0.21902238){\color[rgb]{0,0.61568627,0}\makebox(0,0)[lt]{\lineheight{1.25}\smash{\begin{tabular}[t]{l}$c_u$\end{tabular}}}}%
  \end{picture}%
\endgroup%

	\caption{Geodesic $c_w$ with self-crossing with small crossing angle $\eps =\sphericalangle_p (w, -\dot c_w(T_1)) $
		and its partner orbit $c_u$.}
\end{figure}

\begin{MainThm}\label{thm:partner-general-i}
Let $M$ be a compact Riemannian manifold with sectional curvature $-\kappa_2^2 \leq K \leq -\kappa_1^2$ for $0<\kappa_1 \leq \kappa_2$.
Let $\eps_0:= \frac{\pi}{16}\frac{\kappa_1}{\kappa_1+\kappa_2}$ and let $T_0=t_0(\kappa_1,\kappa_2)>1$ be the one from Prop. \ref{prop:c}.
Then for all closed geodesic $c_w: [0, T] \to M$ of period $T$
with a self-crossing at $p=c_w(0)$ at time $T_1$ and crossing angle $\eps\leq \eps_0$ and $T_1, T-T_1\geq T_0$ there exists a closed orbit $c_u:  [0, T'] \to M$ of period $T'<T$ satisfying
\[
T-T'  \leq \Big(\frac{18}{\kappa_1} + \frac{16\kappa_2}{\kappa_1^2}\Big) \eps,
\]
and such that 
$$
d(  c_u(s), c_w[0, T] )\le \left( \frac{25}{\kappa_1} +\frac{24\kappa_2}{\kappa_1^2} \right) \eps
$$
for all $s \in [0, T^\prime]$.
\end{MainThm}

 For surfaces we can improve the estimate of the action difference obtained in Theorem \ref{thm:partner-general-i}  (see Theorem \ref{thm:partner-orbits-2dim} in Section  \ref{sec:partner-orbits-2dim}).
 In particular, if the pinching constant of the Gaussian curvature is at least $\frac{1}{2}$ the action difference is of order $\eps^2$ in the crossing angle $\eps$.
 This result confirms the findings of Sieber and Richter \cite{SR01, mS02}, and agrees with the result of Huynh and Kunze \cite{HK15} for surfaces of constant negative Gaussian curvature.
 More precisely, we obtain:

\begin{MainThm}\label{thm:partner-orbits-2dim-i}
	Let $(M,g)$ be a  compact Riemannian surface with Gaussian curvature $-\kappa_2^2 \leq K \leq -\kappa_1^2$ for $0<\kappa_1 \leq \kappa_2$.
	Let $\eps_0:= \frac{\pi}{16}\frac{\kappa_1}{\kappa_1+\kappa_2}$ and let $T_0=t_0(\kappa_1,\kappa_2)>1$ be as in Proposition \ref{prop:c}.
	Then for all closed geodesic $c_w\colon[0,T]\rightarrow M$ of period $T$ with a self-crossing at $p =c_w(0)$ at time $T_1$ with crossing angle $\eps\leq \eps_0$
	and $T_1, T-T_1\geq T_0$,
	there exists a closed geodesic $c_u\colon [0,T']\rightarrow M$ of period $T'<T$ such that
	\[
	C_1 \eps^2 \leq T-T'  \leq \left\{
	\begin{array}{ll}
		C_2 \eps^{\frac{4\kappa_1}{\kappa_2}} & \frac{\kappa_1}{\kappa_2}\in\big(0,\frac{1}{2}\big),\\
		C_3 \eps^2 &  \frac{\kappa_1}{\kappa_2}\in\big[\frac{1}{2},1\big],
	\end{array}\right.
	\]
	
	where $C_i$ are suitable positive constants depending on the curvature bounds and the injectivity radius of $M$.
\end{MainThm}

A second consequence of the closing Lemma is the existence of so called pseudo-partner orbits for closed geodesic with large crossing angle,
that is the existence of a pair of two closed geodesics remaining close to the loops of the original self-crossing geodesic.
Their lengths are comparable, although smaller, with the the length of the loops
 (see Theorem \ref{thm:pseudo-partner} in Section \ref{sec:pseudo-partner-orbits}).

\begin{figure}[htb]\label{fig:pseudo-partner}
\begingroup%
  \makeatletter%
  \providecommand\color[2][]{%
    \errmessage{(Inkscape) Color is used for the text in Inkscape, but the package 'color.sty' is not loaded}%
    \renewcommand\color[2][]{}%
  }%
  \providecommand\transparent[1]{%
    \errmessage{(Inkscape) Transparency is used (non-zero) for the text in Inkscape, but the package 'transparent.sty' is not loaded}%
    \renewcommand\transparent[1]{}%
  }%
  \providecommand\rotatebox[2]{#2}%
  \newcommand*\fsize{\dimexpr\f@size pt\relax}%
  \newcommand*\lineheight[1]{\fontsize{\fsize}{#1\fsize}\selectfont}%
  \ifx\svgwidth\undefined%
    \setlength{\unitlength}{206.57437161bp}%
    \ifx\svgscale\undefined%
      \relax%
    \else%
      \setlength{\unitlength}{\unitlength * \real{\svgscale}}%
    \fi%
  \else%
    \setlength{\unitlength}{\svgwidth}%
  \fi%
  \global\let\svgwidth\undefined%
  \global\let\svgscale\undefined%
  \makeatother%
  \begin{picture}(1,0.72840461)%
    \lineheight{1}%
    \setlength\tabcolsep{0pt}%
    \put(0,0){\includegraphics[width=\unitlength,page=1]{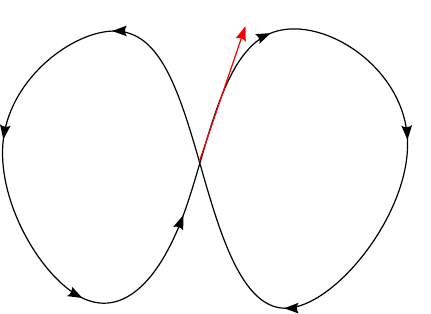}}%
    \put(0.35579859,0.66453263){\makebox(0,0)[lt]{\lineheight{1.25}\smash{\begin{tabular}[t]{l}$w$\end{tabular}}}}%
    \put(0.52824554,0.68284971){\makebox(0,0)[lt]{\lineheight{1.25}\smash{\begin{tabular}[t]{l}$\dot c_w(T_1)$\end{tabular}}}}%
    \put(0,0){\includegraphics[width=\unitlength,page=2]{pseudo-partner-1.pdf}}%
    \put(0.45262396,0.48526027){\makebox(0,0)[lt]{\lineheight{1.25}\smash{\begin{tabular}[t]{l}$\varepsilon$\end{tabular}}}}%
    \put(0,0){\includegraphics[width=\unitlength,page=3]{pseudo-partner-1.pdf}}%
    \put(0.05033039,0.35888056){\color[rgb]{0,0.63529412,0}\makebox(0,0)[lt]{\lineheight{1.25}\smash{\begin{tabular}[t]{l}$c_{u_1}$\end{tabular}}}}%
    \put(0.52502498,0.33415406){\color[rgb]{0,0.63529412,0}\makebox(0,0)[lt]{\lineheight{1.25}\smash{\begin{tabular}[t]{l}$c_{u_2}$\end{tabular}}}}%
    \put(0,0){\includegraphics[width=\unitlength,page=4]{pseudo-partner-1.pdf}}%
  \end{picture}%
\endgroup%

	\caption{Geodesic $c_w$ with a self-crossing of large crossing angle $\sphericalangle_p (w, -\dot c_w(T_1))  \in [\pi -\eps, \pi]$ and its pseudo-partner orbit.}
\end{figure}

Precisely we prove the following:

\begin{MainThm}\label{thm:pseudo-partner-orbits}
		Let $M$ be a compact Riemannian manifold with sectional curvature $-\kappa_2^2 \leq K \leq -\kappa_1^2$ for $0<\kappa_1 \leq \kappa_2$.
	Let $\eps_0:= \frac{\pi}{16}\frac{\kappa_1}{\kappa_1+\kappa_2}$ and let $T_0=t_0(\kappa_1,\kappa_2)>1$ be the one from Prop. \ref{prop:c}.
	Then for all closed geodesic $c_w: [0, T] \to M$ of period $T$
	with a self-crossing at $p=c_w(0)$ at time $T_1$ such that $\sphericalangle_{p}(w, \dot c_{w}(T_1))=\eps\leq \eps_0 $
	and $T_1, T_2:=T-T_1\geq T_0$
	there exists a pair of closed geodesics $c_{u_1}$ and $c_{u_2}$ of period $\hat T_1$ and $\hat T_2$,
	satisfying
	\[
	0\leq T_1-\hat T_1\leq \frac{8}{\kappa_1}\left(\frac{\kappa_2}{\kappa_1}+1\right)\eps
	\;  \;\text{and} \; \;
	0\leq T_2-\hat T_2\leq \frac{8}{\kappa_1}\left(\frac{\kappa_2}{\kappa_1}+1\right)\eps.
	\]
	Furthermore,
	\[
	d(c_{u_1}(s), c_w(s))\leq \frac{12}{\kappa_1}\Big(\frac{\kappa_2}{\kappa_1}+1\Big) \eps
	\]
	for $s\in[0,T_1]$, and
	\[
	d(c_{u_2}(s), c_w(s+T_1))\leq \frac{12}{\kappa_1}\Big(\frac{\kappa_2}{\kappa_1}+1\Big)\eps
	\]
	for $s\in[0,T_2]$.
	In particular,
	\[
	d(c_{u_1}(\hat T_1), c_{u_2}(0))\leq \frac{32}{\kappa_1}\left(\frac{\kappa_2}{\kappa_1}+1\right)\eps.
	\]
	
\end{MainThm}

The paper is organized as follows: in Section \ref{sec:hadamard} we recall important facts about Hadamard manifolds
and we prove some trigonometric estimates which we will need in the sequel. In Section \ref{sec:closing} we obtain the quantitative version of the Anosov closing Lemma formulated in Theorem \ref{thm:closing-i}.
In Section \ref{sec:partner-orbits-general} we derive the existence of Sieber-Richter pairs for any dimension, together with the estimate for the action difference stated in Theorem \ref{thm:partner-general-i},
while in Section \ref{sec:partner-orbits-2dim} we analyse Sieber-Richter pairs for surfaces and obtain an improved estimate for their action difference (see Theorem \ref{thm:partner-orbits-2dim-i}).
Finally in Section \ref{sec:pseudo-partner-orbits} we prove the existence of pseudo-partner orbits (see Theorem \ref{thm:pseudo-partner-orbits}).

\section{Basic facts on Hadamard manifolds}\label{sec:hadamard}
In this section we recall certain standard facts on Hadamard manifolds, i.e, complete simply connected Riemannian manifolds of non-positive curvature which we will need in the sequel.
We start with the construction of the boundary at infinity and a discussion on Busemann functions. We then prove some basic estimates for geodesic triangle under the assumption of
negative pinched curvature and we conclude this section with a quantitative version of the Anosov closing Lemma, on which the construction of partner orbits will be based.

For the rest of this section, $X$ denotes a Hadamard manifold of dimension $d\geq 2$.

\subsection{Boundary at infinity}\label{subsec:boundary}

To introduce the boundary at infinity of $X$, we follow \cite{EO73, Eb1} and we start by defining an equivalence relation on the set of geodesics in $X$.

\begin{definition}
Two geodesic rays  $c_1, c_2\colon [0,\infty) \to X$ are called \emph{asymptotic}
if $ \sup_{t\geq 0} d(c_1(t), c_2(t)) < \infty$.  This is an equivalence relation; we write  $ \partial X$ for the set of equivalence classes and call its elements \emph{points at infinity}. We denote the equivalence class of a geodesic ray (or geodesic) $c$  by $c(\infty)$.
\end{definition}

Given $v\in SX$, let $ c_{v}$ be the unique geodesic with $\dot{c}_{v}(0)=v$.  The following is proved in \cite[Propositions 1.5 and 1.14]{Eb1}.

\begin{lemma}\label{lem:p-to-xi}
Given any $p\in X$ and $\xi\in \partial X$, there is a unique geodesic ray $c=c_{p,\xi}\colon [0,\infty)\to X$ with $c(0) = p$ and $c(\infty) = \xi$.
Equivalently, the map  $f_p\colon S_pX \to \partial X$ defined by $f_p(v) = c_{v}(\infty)$ is a bijection.
\end{lemma}

Following  Eberlein
and O'Neill \cite{EO73} we equip $\partial X$ with a topology that makes it a compact metric space homeomorphic to $S^{n-1}$.  Fix $p\in X$ and let $f_p\colon S_pX  \to \partial X$
 be the bijection $v \mapsto c_{v}(\infty)$ from Lemma \ref{lem:p-to-xi}. The topology (sphere-topology) on $\partial X$ is defined such that $f_p$
becomes a homeomorphism.
Since the map $f_q^{-1} f_p \colon S_p X \to S_q X$ is a homeomorphism for all $q\in X$ \cite{Eb1}, the topology is independent of the reference point $p$.

The topologies on $\partial X$ and $X$
extend naturally
 to  $\bar X: =  X \cup \partial X$
by requiring that the map
$\varphi\colon B_1(p) = \{v \in T_p X:   \|v\| \le 1\} \to \bar X$
defined by
\[
\varphi(v) = \begin{cases}
 \exp_p\left(\frac{v}{1-\|v\|}\right) & \|v\| < 1\\
f_p(v) & \|v\| = 1
\end{cases}
\]
is a homeomorphism. This topology, called the cone topology, was also introduced by Eberlein
and O'Neill in \cite{EO73}. In
particular, $\bar X$ is homeomorphic to a closed ball in
$\mathbb{R}^n$.
The relative topology on $\partial X$ coincides with the sphere topology, and the relative topology on $X$ coincides with the manifold topology.

Every geodesic $c\colon \R \to X$ determines two distinct points $c(-\infty)$ and $c(+\infty)$ on $\partial X$.  We will  use the notation
\begin{equation}\label{eqn:vpm}
v^- := c_{v}(-\infty)
\quad\text{and}\quad
v^+ := c_{v}(+\infty).
\end{equation}
We also write
\begin{equation}\label{eqn:square}
\sqbd:= \{ (\xi,\eta) \in (\partial X)^2 : \xi \neq \eta\}
\end{equation}
and consider the \emph{endpoint map}
\begin{equation}\label{eqn:endpoints}
\begin{aligned}
E \colon SX &\to \sqbd, \qquad
v &\mapsto 
(v^-,v^+).
\end{aligned}
\end{equation}
By \cite[Proposition 1.7]{Eb1}, every pair of distinct points on $\partial X$ is joined by exactly one geodesic, so the map $E$ is onto.
Note that $E$ is continuous by the definition of the topology on $\partial X$.

\subsection{Busemann functions}\label{subsec:busemann}

Given $v\in SX$, the function $b_v: X \to \R$ defined by
\[
b_{v}(q) := \lim_{t\to\infty} \left( d(q, c_{v}(t)) - t\right)
\]
is called the \emph{Busemann function} associated to $v$. Busemann functions are of class $C^2$, as proved by Heintze and Im Hof in (the proof of) \cite[Proposition 3.1]{hi78}, see also the references therein for an
unpublished result of Eberlein.  For a discussion on the regularity of Busemann functions in the more general context of no conjugate points see  \cite{gK85}  and  \cite{gK02}.

\begin{definition}\label{def:busemann}
Given $p\in X$ and $\xi \in \partial X$, let $v\in S_p X$ be the unique unit tangent vector at $p$ such that $c_{v}(\infty) = \xi$.
We call $b_\xi(q,p) := b_{v}(q)$ the Busemann function based at $\xi$
and normalized by $p$ ($b_\xi(p,p) =0$).
\end{definition}

The zero set of the Busemann function  $b_\xi(q,p)$ given by
$$
 H_\xi(p) := \{q \in X \mid b_\xi(q,p) = 0\}
 $$
is called the horosphere through $p$ centred at $\xi$.
 The  Busemann function fulfils the following \emph{cocycle property}:
\begin{equation}\label{eqn:pqr}
b_\xi(p,q) = b_\xi(p,r) + b_\xi(r,q)
\text{ for all } \xi\in \partial X \text{ and } p,q,r\in X.
\end{equation}
Note also that
\begin{equation}\label{eqn:pq}
b_\xi(p,q) = -b_\xi(q,p) \text{ for all } \xi\in \partial X \text{ and } p,q\in X.
\end{equation}
Furthermore, $\grad b_\xi(q,p) = -\dot{c}_{q,\xi}(0)$ and is, in particular, independent of $p$.
 For $q \in X$ and $\xi, \eta \in \partial X$ we denote by
$$
\sphericalangle_q(\xi,\eta) = \sphericalangle_q( \dot{c}_{q,\xi}(0), \dot{c}_{q,\eta}(0))= \cos \langle  \grad b_\xi(q,p), \grad b_\eta(q,p) \rangle
$$
the angle of $\xi, \eta $ as seen from $q$ (visibility angle).
\begin{lemma} \label{lem:angle}
Let $X$ be a Hadamard manifold.  Then for any pair of points $\xi, \eta \in \partial X$ the visibility angle $f(q) =\sphericalangle_q(\xi,\eta)$ defines a $C^1$ function
on $q\in X$. If  the sectional curvature $K$ of $X$ satisfies $-\kappa^2 \le K $  for some $\kappa \ge0$, then the norm of the gradient of $f$ is bounded by $\kappa$.
In particular, $f$ is Lipschitz continuous with Lipschitz constant $\kappa$.
\end{lemma}

\begin{proof}
	Since Busemann functions are $C^2$, it follows from the definition of the visibility angle that $f$ is a $C^1$ function.
Since the Hessian of the Busemann function is positive semi-definite and bounded from above by $\kappa$, we obtain the estimate for the gradient of $f$.
\end{proof}

\subsection{Some trigonometrical estimates for pinched negatively curved Hadamard manifolds}\label{sec:comparison}

In this subsection we present some results related to geodesic triangles, angles, and orthogonal projections in
a Hadamard manifold $X$ with bounded sectional curvature $K$.

In the following, we denote by $\Delta(p_1,p_2,p_3)$ a geodesic triangle with vertices $p_1,p_2,p_3$.
We may also use the notation $\Delta_\bullet(p_1,p_2,p_3)$ to specify in which manifold the triangle lives.
The angles with vertices $p_i$ are denoted by $\alpha_i$ and the lengths of the sides opposite to $\alpha_i$ by $\ell_i$, $i=1,2,3$.

We start by recalling the law of cosine and sine in the case of non-constant negative curvature.

\begin{lemma}\label{lem:tri1}
	Let $X$ be a Hadamard manifold with $-\kappa_2^2 \leq K \leq -\kappa_1^2$ for some $0 < \kappa_1 \leq \kappa_2 <\infty$,
	and let $\Delta(p_1,p_2,p_3)\subset X$ as described above.
	Then
	\begin{equation}\label{eq:lawcosin-kappa1}
		\cosh(\kappa_1 \ell_3) \ge \cosh(\kappa_1 \ell_1)\cosh(\kappa_1 \ell_2) - \sinh(\kappa_1 \ell_1)\sinh(\kappa_1 \ell_2) \cos (\alpha_3),
	\end{equation}
	and
	\begin{equation}\label{eq:lawcosine-kappa2}
		\cosh(\kappa_2 \ell_3) \le \cosh(\kappa_2 \ell_1)\cosh(\kappa_2 \ell_2) - \sinh(\kappa_2 \ell_1)\sinh(\kappa_2 \ell_2) \cos (\alpha_3).
	\end{equation}
	Moreover, if $\alpha_3 = \frac{\pi}{2}$ we have
	\begin{equation}\label{eq:lawsin-kappa1}
		\sin(\alpha_i) \le \frac{\sinh(\kappa_1 \ell_i)}{\sinh(\kappa_1 \ell_3)} \qquad i=1,2,
	\end{equation}
	and
	\begin{equation}\label{eq:lawsin-kappa_2}
		\sin(\alpha_i) \ge \frac{\sinh(\kappa_2 \ell_i)}{\sinh(\kappa_2 \ell_3)} \qquad i=1,2.
	\end{equation}
	Equality in all of the above holds if $\kappa_1=\kappa_2$.
\end{lemma}

The above inequalities give us the next lemma.

\begin{lemma}\label{lem:tri2}
Let $X$ be a Hadamard manifold with $K\leq -\kappa_1^2$ for some $\kappa_1>0$ and let $\Delta(p_1 p_2 p_3)$ be as in the beginning of this subsection.
	Then
	\begin{equation}\label{eq:angle-bound}
	\frac{2\alpha_3^2}{\pi^2} \le 2 \sin ^2(\frac{\alpha_3}{2})     \leq \frac{\cosh(\kappa_1 \ell_3) -1}{ \sinh(\kappa_1 \ell_1)\sinh(\kappa_1 \ell_2)}.
	\end{equation}
	Furthermore assume that $\alpha_3= \frac{\pi}{2}$,then
	\begin{equation}\label{eq:cosh-bound}
	\cosh(\kappa_1 \ell_1) \le \frac{1}{\sin (\alpha_2)}
	\end{equation}
	and
	\begin{equation}\label{eq:sin-bound}
	\sin (\alpha_1) \le \frac{\cot(\alpha_2)}{\sinh (\kappa_1 \ell_3)}.
	\end{equation}	
	
	\end{lemma}

\begin{proof}
We start by proving \eqref{eq:angle-bound}. The addition theorem for hyperbolic functions yields
$$
 \cosh(\kappa_1 \ell_1)\cosh(\kappa_1 \ell_2) =  \sinh(\kappa_1 \ell_1)\sinh(\kappa_1 \ell_2) +  \cosh(\kappa_1( \ell_2- \ell_1)).
 $$
Therefore, by the law of cosine  \eqref{eq:lawcosin-kappa1} we obtain
\begin{align*}
\cosh(\kappa_1 \ell_3)& \ge  \sinh(\kappa_1 \ell_1)\sinh(\kappa_1 \ell_2)(1- \cos (\alpha_3)) +  \cosh(\kappa_1( \ell_2- \ell_1))\\
& \ge \sinh(\kappa_1 \ell_1)\sinh(\kappa_1 \ell_2)(1- \cos (\alpha_3)) +1.
\end{align*}
Using $ 1- \cos (\alpha_3) = 2 \sin^2 ( \frac{\alpha_3}{2})$ and the inequality  $\sin \alpha \ge \frac{2}{\pi}\alpha $ for all $\alpha \in [0, \frac{\pi}{2}]$
we obtain \eqref{eq:angle-bound}.

We now turn to the proof of \eqref{eq:cosh-bound} and \eqref{eq:sin-bound}. The law of cosine \eqref{eq:lawcosin-kappa1} with $\alpha_3=\pi/2$ and \eqref{eq:lawsin-kappa1} 	
	yield
	\begin{align*}
		\cosh^2(\kappa_1 \ell_1) & \le \frac{\cosh^2(\kappa_1 \ell_3)}{\cosh^2(\kappa_1 \ell_2)} =\frac{\cosh^2(\kappa_1 \ell_3)}{1 +\sinh^2(\kappa_1 \ell_2)} \\
		&\le \frac{\cosh^2(\kappa_1 \ell_3)}{1 + \sin^2( \alpha_2)\sinh^2(\kappa_1 \ell_3)} =\frac{1+\sinh^2(\kappa_1 \ell_3)}{1 + \sin^2( \alpha_2)\sinh^2(\kappa_1 \ell_3)}\\
		&\le \frac{1}{\sin^2(\alpha_2)},
	\end{align*}
	which proves \eqref{eq:cosh-bound}. Equation \eqref{eq:sin-bound} is a consequence of \eqref{eq:lawsin-kappa1} and the estimate just proved:
	\begin{align*}
		\sin^2 (\alpha_1) &\le  \frac{\sinh^2(\kappa_1  \ell_1)}{\sinh^2( \kappa_1 \ell_3)} = \frac{\cosh^2(\kappa_1  \ell_1) -1}{\sinh^2( \kappa_1 \ell_3)}\\
		&\le \frac{\frac{1}{\sin^2(\alpha_2)} -1}{\sinh^2( \kappa_1 \ell_3)} = \frac{1 - \sin^2(\alpha_2)}{\sin^2(\alpha_2)\sinh^2(\kappa_1 \ell_3)} \\
		&= \frac{\cos^2(\alpha_2)}{\sin^2(\alpha_2)\sinh^2(\kappa_1 \ell_3)}=  \frac{\cot^2(\alpha_2)}{\sinh^2 (\kappa_1 \ell_3)}.
	\end{align*}
\end{proof}

The corollary below gives relations among sides, angles, and distances from a vertex to the opposite side in a geodesic triangle, assuming
that one of the angles is large enough and the lengths of the sides starting at its vertex are bounded from below.

\begin{corollary}\label{cor:tri}
	Let $X$ be a Hadamard manifold with $K\leq -\kappa_1^2$ for some $\kappa_1>0$, $\Delta(p_1, p_2, p_3)$ be as in  at the beginning of this subsection,
	and assume $\alpha_3 \in [\pi - \eps,\pi ]$ for some $\eps\in(0,\frac{\pi}{2})$.
	Set $a(\eps)= \frac{1}{\kappa_1} \arcosh \left(\frac{1}{\cos (\eps)}  \right)$ and denote further by $c_1$ the geodesic
	connecting $p_3$ and $p_2$, by $c_2$ the geodesic connecting $p_3$ and $p_1$,
	and by $c_3$ the geodesic connecting $p_1$ and $p_2$.
	Then
	\begin{equation}\label{eq:distance1}
	c_1([0, \ell_1]), c_2([0, \ell_1])\in \left\{q\in X \,\vert\, d(q,c_3)\le a\big( \frac{\eps}{2}\big) \right\}.
	\end{equation}
	Assume furthermore that $\ell_i \ge R_i$ for some $R_i>0$, $i=1,2$, then
	\begin{equation}\label{eq:distance2}
	\ell_3 \ge R_1+R_2 -2a\Big(\frac{\eps}{2}\Big),
	\end{equation}
	and
	\begin{equation}\label{eq:sin}
	\sin(\alpha_i) \le \tan(\eps)\sinh^{-1}(\kappa_1 R_{j}) \quad i,j\in \{1,2\}, i\neq j.
	\end{equation}	
\end{corollary}

\begin{proof}
	We first observe that the functions $t \mapsto d(c_1(t) , c_3)$ and $t \mapsto  d(c_2(t) , c_3)$  are convex and therefore
	$$
	\max_{q \in c_1\cup c_2} d(q, c_3)= d (p_3,c_3).
	$$
	Since $\alpha_3 \ge \pi/2$, the foot-point projection  $p$  of $p_3$ onto $c_3$  is contained in the geodesic segment between $p_1$ and $p_2$.
	Consider the angles
	$\varphi_1 =\sphericalangle_{p_3}(p_1, p)$ and $ \varphi_2 =\sphericalangle_{p_3}(p_2, p)$.
	Since at least one of the angles $\varphi_i$ is bigger or equal than
	$\frac{\alpha_3}{2}$, and $\frac{\pi}{2} - \frac{\eps}{2} \le \frac{ \alpha_3}{2}  \le \frac{\pi}{2}$, Eq. \eqref{eq:cosh-bound} in Lemma \ref{lem:tri2} implies
	\begin{align*}
	d(p_3,c_3) & =d(p_3,p) \le \frac{1}{\kappa_1}\arcosh\left( \frac{1}{\sin(\varphi_i)}\right) \\
	 & \leq \frac{1}{\kappa_1} \arcosh \left(\frac{1}{\sin \left(\frac{ \alpha_3}{2} \right)}  \right) \le  \frac{1}{\kappa_1} \arcosh \left(\frac{1}{ \cos \left(\frac{\eps}{2}\right)}   \right) = a\big(\frac{\eps}{2}\big),
	\end{align*}
	proving \eqref{eq:distance1}. The equation \eqref{eq:distance2} is a direct consequence of the triangle inequality:
	\[
	\ell_3 = d(p_1,p) +d(p,p_2)\geq R_1 -a(\eps/2) +R_2-a(\eps/2).
	\]
	
	We now turn to the proof of \eqref{eq:sin}. Eq. \eqref{eq:sin-bound} applied to $\Delta(p_i, p_3,p)$, $i=1,2$, yields
	\[
	\sin(\alpha_i) \le \frac{\cot (\varphi_i)}{\sinh(\kappa_1 R_{j} )} 
	\]
	for $i,j\in \{1,2\}$ with $i\neq j$. Since $\pi - \eps \le \varphi_1+ \varphi_2 \le \pi $ and $ \varphi_i \le \frac{\pi}{2}$, we obtain
	$\varphi_2 \ge \pi  - \eps - \varphi_1 \ge \frac{\pi}{2} - \eps $ and similarly $\varphi_1 \ge  \frac{\pi}{2} - \eps $.
	Therefore, we obtain
	\[
	\cot (\varphi_i) \le \cot(\frac{\pi}{2} - \eps) = \tan \eps,
	\]
	which implies the claim.
\end{proof}

%
%

The next lemma tells us that the angle between points at infinity is small if their corresponding vectors are close w.r.t. to the metric on $SX$ defined by
$$
d_{1}(v,w)=\max_{t\in[-1,1]} d(c_v(t), c_w(t)),
$$
where the metric $d$ is induced by the Riemannian metric on $X$. Note that by definition the metric is flip-invariant, i.e. $d_{1}(v,w) = d_{1}(-v,-w)$.

\begin{lemma}\label{lemma:angles-at-infinity}
	Let $X$ be a Hadamard manifold with sectional curvature $ K \leq -\kappa_1^2$ for some $\kappa_1 > 0$ and let $SX$ be endowed with the metric $d_1$ defined above.
	Let $v,w\in SX$ such that $d_{1}(v,w) \leq \delta$ for $\delta\in \big(0,\frac{1}{2}\big]$.
	Then for $p = \pi w$ we have
	\[
	\sphericalangle_{p}(c_v(+\infty), c_w(+\infty) ) \leq  f(\delta) \; \text {and } \; \sphericalangle_{p}(c_v(-\infty), c_w(-\infty) ) \leq  f(\delta),
	\]
	where $ f(\delta) =  2\arcsin\Big( \frac{\sinh(\kappa_1\delta)}{\sinh(\kappa_1(1-\delta))}\Big)$.
\end{lemma}

\begin{remark}\label{rmk:angle-at-infinity-approx}
	Since the function $f(\delta)$
	is convex and attains its maximum at $\frac{1}{2}$ with value $\pi$,
	it holds
	\[
	f(\delta)\leq 2\pi \delta \qquad \forall\delta\in\Big(0,\frac{1}{2}\Big].
	\]
\end{remark}
\begin{proof}
	By the flip-invariance of $d_1$ it is enough to prove  $\sphericalangle_{p}(c_v(+\infty), c_w(+\infty) ) \leq  f(\delta)$.
	Let $s>1$ and set $\alpha(s):=\sphericalangle_{c_w(0)}(c_w(+\infty), c_v(s))$, $\beta:=\sphericalangle_{c_w(0)}(c_v(1), c_w(1))$ and
	$\tilde\alpha(s)=\sphericalangle_{c_w(0)}(c_v(1), c_v(s))$.  Then
	\[
	\sphericalangle_{p}(c_v(+\infty), c_w(+\infty) )=\lim_{s\to\infty} \alpha(s) \leq \lim_{s\to\infty} \tilde\alpha(s) +\beta.
	\]
	It is enough to show that $\beta+\tilde \alpha(s)$ is bounded by $f(\delta):=2\arcsin\Big( \frac{\sinh(\kappa_1\delta)}{\sinh(\kappa_1(1-\delta))}\Big)$.
	
	By assumption we have $d(c_v(0), c_w(0)) \leq \delta$ and $d(c_v(1), c_w(1)) \leq \delta$.
	Let $q$ be the foot-point projection of $c_v(1)$ onto $c_w$ and consider the triangle $\Delta(c_w(0), q, c_v(1))$.
	Then we have $\beta=\sphericalangle_{c_w(0)}(c_v(1), q)$ and by \eqref{eq:lawsin-kappa1} we have
	\[
	\sin(\beta)\leq \frac{\sinh(\kappa_1 d(c_v(1), q))}{\sinh(\kappa_1 d(c_w(0), c_v(1)))}.
	\]
	By the triangle inequality we have
	\[
	d(c_w(0), c_v(1))\geq  d(c_w(0), c_w(1)) - d(c_w(1), c_v(1)) \geq 1-\delta
	\]
	and therefore we obtain
	\[
	\sin(\beta)\leq \frac{\sinh(\kappa_1\delta)}{\sinh(\kappa_1(1-\delta))},
	\]
	which implies $\beta\leq \frac{1}{2}f(\delta)$.

For the angle $\eta:=\sphericalangle_{c_v(1)}(c_w(0), c_v(0))$, we have
$\eta \leq \frac{1}{2} f(\delta)$,
	with a similar argument as before considering the triangle $\Delta(c_w(0), c_v(1), q_1)$, where $q_1$ is the foot-point projection of $c_w(0)$ onto $c_v$.
	
	Consider now the triangle $\Delta(c_w(0), c_v(s), c_v(1))$ with angles $\tilde\alpha(s)$, $\eta_1$ and $\gamma$ at the points $c_w(0), c_v(1)$ and $c_v(s)$, respectively.
	Since $\eta_1$ is the complement of the angle $\eta$ and $\tilde\alpha(s)+\eta_1+\gamma\leq \pi$, we have $\pi-\eta+\tilde\alpha(s)+\gamma\leq \pi$,
	which implies $\tilde\alpha(s)\leq \eta$.
	
	We conclude
	\[
	\beta+\tilde\alpha(s)\leq f(\delta)
	\]
	for all $s>1$ and the claim follows.
\end{proof}

We conclude this subsection with a result for orthogonal projections on geodesics in $2$-dimensional Hadamard manifolds.

\begin{figure}[htb]\label{fig:polygons-comparison}
	\def\svgwidth{\linewidth}
	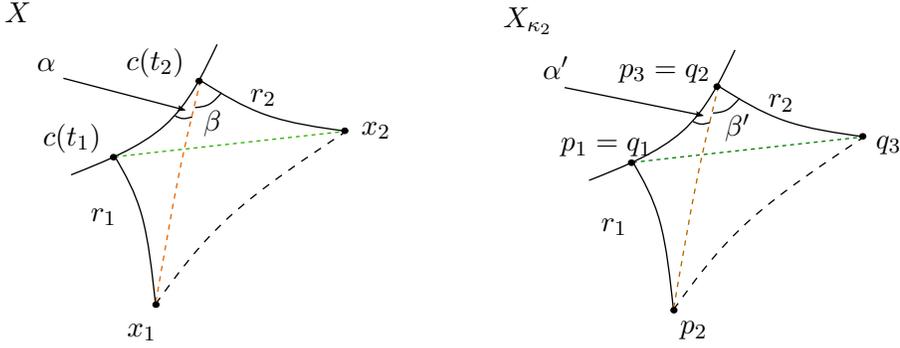
	\caption{The $4$-gons in Lemma \ref{lem:perp-proj}.}
\end{figure}

\begin{lemma}\label{lem:perp-proj}
	Let $X$ be a Hadamard manifold of dimension $2$ with sectional curvature $ -\kappa_2^2 \leq K < 0$ for some $\kappa_2>0$.
	Let $c$ be a geodesic in $M$ and let $x_1\neq x_2\in X$ not lying on $c$ and such that the geodesic connecting $x_1$ and $x_2$ does not cross the geodesic $c$.
	Let further $r_i= d(x_i, c)=d(x_i, c(t_i))$ for $i=1,2$, then
	\begin{multline}\label{eq:perp-proj}
		\sinh^2\big(\kappa_2 \frac{d(x_1,x_2)}{2}\big) \leq \cosh(\kappa_2 r_1)\cosh(\kappa_2 r_2) \sinh^2\big( \kappa_2\frac{\vert t_2-t_1\vert}{2}\big)\\
		+\sinh^2\big(\kappa_2 \frac{\vert r_2-r_1\vert}{2}\big),
	\end{multline}
	with equality if $K=\kappa_2$.
\end{lemma}

\begin{proof}
	We first consider the case $K=-\kappa^2=-\kappa_2^2$.
	The general case follows by a comparison argument.
	
	Assume that $t_2> t_1$. Set $s:= d(x_1, c(t_2))$ and let $\alpha= \sphericalangle_{c(t_2)}(c(t_1), x_1)$.
	Since the angle at $c(t_1)$ is equal to $\pi/2$ we have
	\[
	\cosh(\kappa s) = \cosh(\kappa (t_2-t_1))\cosh(\kappa r_1).
	\]
	Since $\dim X = 2$ we have $\beta:= \sphericalangle_{c(t_2)}(x_2,x_1) = \pi/2 -\alpha$ and so
	\[
	\cos(\beta)= \sin(\alpha)= \frac{\sinh(\kappa r_1)}{\sinh(\kappa s)}
	\]
	by law of sine in negative constant curvature.
	By the law of cosine we also have
	\[
	\cos(\beta)= \frac{\cosh(\kappa s)\cosh(\kappa r_2) - \cosh(\kappa d(x_1,x_2))}{\sinh(\kappa s)\sinh(\kappa r_2)}.
	\]
	We therefore obtain
	\begin{align*}
		\sinh(\kappa r_1)\sinh(\kappa r_2) & = \cosh(\kappa s)\cosh(\kappa r_2) - \cosh(\kappa d(x_1,x_2))\\
		& = \cosh(\kappa r_1) \cosh(\kappa r_2)\cosh(\kappa(t_2-t_1)) -\cosh(\kappa d(x_1,x_2)).
	\end{align*}
	Using $\cosh(x)=2\sinh^2(x/2) +1$ for $x=\kappa (t_2-t_1)$ and $x= \kappa d(x_1,x_2)$ we obtain
	\begin{multline*}
		\sinh(\kappa r_1)\sinh(\kappa r_2) = 2\cosh(\kappa r_1) \cosh(\kappa r_2)\sinh^2\big( \kappa \frac{t_2-t_1}{2}\big)\\
		+\cosh(\kappa r_1) \cosh(\kappa r_2)	-2\sinh^2\big( \kappa \frac{d(x_1,x_2)}{2}\big) -1.
	\end{multline*}
	Since $\cosh( \kappa \vert r_2-r_1\vert) = \cosh(\kappa r_1)\cosh(\kappa r_2) -\sinh(\kappa r_1)\sinh(\kappa r_2)$, and using again
	$\cosh(\kappa \vert r_2-r_1\vert)=2\sinh(\kappa \vert r_2-r_1\vert/2) +1$ the claim follows.
	
	\medskip
	
	Let now $X$ be with sectional curvature $ -\kappa_2^2 \leq K \leq 0$. Consider the triangles $\Delta_X( c(t_1), x_1, c(t_2))$ and $\Delta_X(c(t_2), x_2, c(t_1))$.
	Denote by $X_{\kappa_2}$ the Hadamard manifold of the same dimension with sectional curvature equal to $-\kappa_2^2$. Consider the triangle
	$\Delta_{X_{\kappa_2}}( p_1,p_2,p_3)$ such that
	\[
	\sphericalangle p_1= \sphericalangle c(t_1)=\pi/2, \quad d_{X_{\kappa_2}}(p_1,p_2)= d_X(c(t_1), x_1)=r_1, \quad d_{X_{\kappa_2}}(p_1,p_3)= t_2-t_1
	\]
	(we again assume $t_2>t_1$). By Toponogov's theorem
	\[
	d_{X_{\kappa_2}}(p_2,p_3)\geq d_X(c(t_2),x_1).
	\]
	Similarly consider the triangle $\Delta_{X_{\kappa_2}}(q_1, q_2, q_3)$ such that
	\[
	\sphericalangle q_2=\sphericalangle c(t_2), \quad d_{X_{\kappa_2}}(q_2,q_1)=t_2-t_1, \quad d_{X_{\kappa_2}}(q_2, q_3)= d_X(c(t_2), x_2)=r_2.
	\]
	Again we have
	\[
	d_{X_{\kappa_2}}(q_1,q_3)\geq d_X(c(t_1), x_2).
	\]	
	Glue $\Delta_{X_{\kappa_2}}(q_1,q_2,q_3)$ and $\Delta_{X_{\kappa_2}}(p_1,p_2,p_3)$ along the edge with common length
	such that $q_2=p_3$ and $q_1=p_1$ and so that the points $q_2=p_3$, $q_1=p_1$, $p_2$ and $q_3$ are vertices of a $4$-gon in $X_{\kappa_2}$,
	see Figure \ref{fig:polygons-comparison}.
		
	Let now $\alpha'=\sphericalangle_{p_3} (p_1,p_2)$ and $\alpha=\sphericalangle_{c(t_2)}(c(t_1),x_1)$. We have
	\[
	\sin(\alpha')= \frac{\sinh(\kappa_2 r_1)}{\sinh(\kappa_2 d_{X_{\kappa_2}}(p_2,p_3))}
	\leq \frac{\sinh(\kappa_2 r_1)}{\sinh(\kappa_2 d_X(c(t_2),x_1)}\leq \sin(\alpha).
	\]
	Since $\alpha', \alpha < \pi/2$ and sine is increasing on $[0,\pi/2]$, we conclude $\alpha' \leq \alpha$.
	Setting $\beta = \sphericalangle_{c(t_2)}(x_1,x_2)$ and $\beta'= \sphericalangle_{q_2}(p_2,q_3)$,
	we have $\beta'=\pi/2-\alpha'\geq \pi/2-\alpha=\beta$.
	Therefore
	\begin{align*}
		\cosh(\kappa_2 d_{X_{\kappa_2}}(p_2,q_3)) &= \cosh(\kappa_2 r_2)\cosh(\kappa_2 d_{X_{\kappa_2}}(p_2,p_3)) \\
		&\hspace{1cm}-\sinh(\kappa_2 r_2)\sinh(\kappa_2 d_{X_{\kappa_2}}(p_2,p_3))\cos(\beta')\\
		& \geq 	\cosh(\kappa_2 r_2)\cosh(\kappa_2 d_{X}(c(t_2), x_1))\\
		&\hspace{1cm} -\sinh(\kappa_2 r_2)\sinh(\kappa_2 d_{X}(c(t_2), x_1))\cos(\beta)\\
		&\geq \cosh( \kappa_2 d_X(x_1,x_2)),
	\end{align*}
	and we conclude $d_{X_{\kappa_2}}(p_2,q_3)\geq d_X (x_1,x_2)$.
	Hence
	\begin{align*}
		\sinh^2\big( \kappa_2 \frac{d_X(x_1,x_2)}{2}\big) &\geq \sinh^2(\kappa_2 d_{X_{\kappa_2}}(p_2,q_3)/2)\\
		& = \cosh(\kappa_2 r_1)\cosh(\kappa_2 r_2) \sinh^2\big( \kappa_2 \frac{t_2-t_1}{2}\big) \\
		&\hspace{3cm}+\sinh^2\big(\kappa_2\frac{\vert r_2-r_1\vert}{2}\big),
	\end{align*}	
as claimed.
\end{proof}

\section{A quantitative closing lemma}\label{sec:closing}

Let $M= X / \Gamma$ be a closed $d$-dimensional manifold, $d\geq 2$, with sectional curvature
$- \kappa_2^2 \le K \le - \kappa_1^2$ for some $0 < \kappa_1 \le  \kappa_2 < \infty$,
where $X$ is its universal cover and $\Gamma$ is the group of deck transformations.
The map $\pr\colon X \rightarrow M$ denotes the canonical projection,
while we denote with $\pi: SM \rightarrow M$ the foot-point projection defined by  $\pi(v) = p$ for all $v\in S_pM$.
Furthermore, we will denote by $\phi^t$ the geodesic flow on $M$, i.e. the map $\phi^t\colon SM \rightarrow SM$, $\phi^t(v)= \dot c_v(t)$,
where $c_v$ is the unique geodesic with $c_v(0)=\pi(v)$ and $\dot c_v(0)=v$.
By abuse of notation, we will also denote by $\pi$ and $\phi^t$ the foot-point projection from $SX$ onto $X$ and the geodesic flow on $X$, respectively.

The class of closed geodesics on $M$ is in one-to-one correspondence with the conjugacy classes in the group $\Gamma$ of deck transformations.
We here recall some elements of this correspondence, which we will use in the next sections.

Given a closed geodesic $\bar c \colon \R/\ell \Z \to M$, the length $\ell>0$ is such that $\bar c(t+\ell) = \bar c(t)$ for all $t\in\R$, and thus for every lift $c$ of $\bar c$ there is a unique $\gamma\in\Gamma$ such that
\begin{equation}\label{eqn:axis}
c(t+\ell) = \gamma c(t) \text{ for all }t\in\R.
\end{equation}

\begin{definition}
If $c$, $\gamma$, and $\ell>0$ are such that \eqref{eqn:axis} holds, then we say that $\gamma$ is the \emph{axial isometry} of $(c,\ell)$, and $c$ is an \emph{axis} of $\gamma$.
\end{definition}
If $c$ is an axis of some $\gamma$, then it follows immediately that $\bar c = {\pr}\circ c$ is a closed geodesic on $M$. Moreover, every axial isometry $\gamma$ of $c$ fixes $c(\pm\infty)$.

\begin{lemma}\label{lem:Preissmann}
Given any geodesic $c$ on $X$ such that $\bar c = {\pr}\circ c$ is closed, the set of $\gamma\in \Gamma$ fixing $c(-\infty)$ and $c(\infty)$ is an infinite cyclic subgroup.
\end{lemma}

The preceding discussion lets us go from closed geodesics to deck transformations. We will also need to go in the other direction.

\begin{definition}
Given $\gamma\in\Gamma\setminus\id$, the \emph{length} of $\gamma$ is
\[
|\gamma|:=\inf\{d(q,\gamma q): q\in X\}.
\]
\end{definition}

\begin{lemma}\label{lem:ax}
For every $\gamma\in\Gamma\setminus\id$, there exists $q_0\in X$ such that $d(q_0,\gamma q_0) = |\gamma| \geq 2\inj(M)$. Moreover, if $c\colon \R\to X$ is a geodesic joining $q_0$ and $\gamma q_0$, then $c(t+|\gamma|) = \gamma c(t)$ for all $t\in\R$, so
$\bar c := {\pr}\circ c$ is a closed geodesic.
\end{lemma}

Using Busemann functions we can define the so called Hopf map.
\begin{definition}\label{def:hopf}
The \emph{Hopf map} $H \colon SX \to \sqbd \times \R$ for $p_0\in X$ is
\begin{equation}\label{eqn:hopf}
H(v) := (v^-, v^+, s(v)),
\text{ where }
s(v) :=b_{v^-}(\pi v, p_0).
\end{equation}
\end{definition}
%

%

Given disjoint sets $\Pa,\Fu\subset \partial X$, the set $H^{-1}(\Pa\times\Fu\times\{0\})$ represents the set of all $v\in SX$ whose past history under $\phi^t$ is given by $\Pa$, whose future evolution is given by $\Fu$, and
such that $b_{v^-}(\pi v, p_0) = 0$.
We will use the following choice of $\Pa,\Fu$: given our fixed choice of $v_0 \in S_{p_0} X $, we consider for each $\theta>0$ the sets
\begin{equation}\label{eqn:Cpm}
\begin{aligned}
\Pa = \Pa_\theta(v_0) &:= \{ w^- \mid w\in S_{p_0} X \text{ and } \sphericalangle_{p_0}(w,v_0) \leq \theta \}, \\
\Fu = \Fu_\theta(v_0) &:= \{ w^+  \mid w\in S_{p_0}X \text{ and } \sphericalangle_{p_0}(w, v_0) \leq \theta \}.
\end{aligned}
\end{equation}

\begin{lemma}\label{lem:dis}
Let $c: \R \to X$ be a geodesic with $c(\infty) \in \Fu_\theta(v_0) $ and $c(-\infty) \in \Pa_\theta(v_0) $ and $p_0 = \pi v_0 $.Then
\[
d(p_0,c) \le a(\theta):=  \frac{1}{\kappa_1} \arcosh \left(\frac{1}{\cos (\theta)} \right).
\]
\end{lemma}

\begin{proof}
By assumption there are $v, w \in S_{p_0}X$ such that $\sphericalangle_{p_0}(v,v_0) \leq \theta $, $\sphericalangle_{p_0}(w,-v_0) \leq \theta $,
$v^+ = c( \infty)$, and $w^+ = c( -\infty)$. Then $ \sphericalangle_{p_0}(v,w) \ge \pi - 2 \theta$ and Corollary \ref{cor:tri} yields the claim.
\end{proof}

Now consider
\begin{equation}\label{eqn:G*}
 \Gamma^*_\theta(v_0) := \{\gamma \in \Gamma \mid
\gamma \Fu_\theta (v_0)\subset \Fu_\theta(v_0) \text{ and } \gamma^{-1} \Pa_\theta(v_0) \subset \Pa_\theta(v_0) \}.
\end{equation}

\begin{lemma}\label{lem:closing-0}
Given any $\gamma\in  \Gamma^*_\theta(v_0)$, there exists an axis $c$ for $\gamma$ such that $c(-\infty) \in \Pa_\theta(v_0) $ and $c(+\infty) \in \Fu_\theta(v_0) $.
\end{lemma}

\begin{proof}
By the Brouwer fixed point theorem, $\gamma$ has one fixed point in $\Pa_\theta(v_0)$ and one in $\Fu_\theta(v_0)$.
By Lemma \ref{lem:ax}, it has up to parametrization a unique axis $c$ whose endpoints are the fixed points of $\gamma$.
\end{proof}

The following definition is justified by Lemma \ref{lem:dis}.

\begin{multline}\label{theta}
A_\theta(v_0) := \{v\in SX \mid (v^-, v^+) \in \Pa_\theta(v_0)\times\Fu_\theta(v_0) \\
\text{ and } d(\pi v, \pi v_0) \le a(\theta) \}.
\end{multline}
Further we define
\begin{equation}\label{eqn:Gt}
\Gamma_\theta(v_0, t) :=
\{\gamma \in \Gamma \mid A_\theta(v_0)  \cap  \phi^{-t} \gamma_{\ast} A_\theta(v_0) \neq \emptyset \}.
\end{equation}

The next proposition will be crucial for the proof of the quantitative closing Lemma.
Its proof is inspired by the one given in \cite[Lemma 4.9]{CKW}.

\begin{proposition}\label{prop:c}
	Consider the constant  $ \theta_0 =\frac{\pi}{8} \frac{\kappa_1}{\kappa_2+ \kappa_1}  $ and the function $\rho(\theta):= 2(\frac{\kappa_2}{\kappa_1}+1)\theta$.
	Then there exists $t_0 =t_0(\kappa_1, \kappa_2)>0$ such that for all $\theta  \in [0, \theta_0]$, $t \ge t_0$ and all  $\gamma \in \Gamma_\theta(v_0, t)$ we have
	
	\[
	\gamma \Fu_{\rho(\theta)  }(v_0) \subset \Fu_{\rho(\theta)  }(v_0)  \;\; \text{and}   \;\; \gamma^{-1} \Pa_{\rho(\theta)  }(v_0)  \subset \Pa_{\rho(\theta)  }(v_0).
	\]
	In particular, $\Gamma_\theta(v_0, t)\subset \Gamma^*_{\rho(\theta)}(v_0)$.
\end{proposition}
\begin{remark}
It is possible to determine $t_0$  explicitly. In particular, $t_0>1$.
\end{remark}

\begin{figure}[htb]
	\def\svgwidth{\linewidth}
	\scalebox{.7}{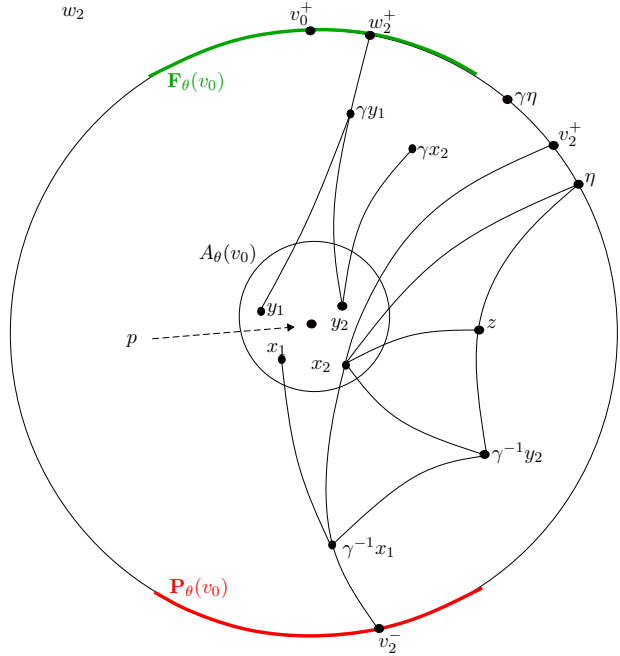}
	\caption{Proving the inclusion $\gamma \Fu_{\rho(\theta,t)} \subset \Fu_{\rho(\theta,t)}$ in proposition \ref{prop:c}.}
\end{figure}

\begin{proof}
	
	We prove the first inclusion; the second is similar.
	Consider $\gamma \in \Gamma_\theta(v_0, t) $ for some $t>0$. Then
	the set  $A_\theta (v_0) \cap  \phi^{-t} \gamma_{\ast} A_\theta(v_0)$ is non-empty and therefore the sets
	$$
	\gamma_\ast^{-1}A_\theta (v_0) \cap \phi^{-t}  A_\theta(v_0) \; \;\text{and}  \; \; \phi^{t}  A_\theta(v_0) \cap  \gamma_{\ast} A_\theta (v_0)
	$$
	are non-empty as well. Therefore,
	there exist $v_1, w_1 \in A_\theta(v_0) $ with $\gamma_\ast^{-1}v_1 \in \phi^{-t}  A_\theta(v_0)$ and
	$\gamma_\ast w_1 \in  \phi^{t}  A_\theta(v_0)$ and consequently there exist $v_2, w_2 \in  A_\theta(v_0)$ with $\gamma_\ast^{-1}v_1 =  \phi^{-t}  v_2$
	and  $\gamma_\ast w_1 = \phi^{t} w_2$. From the definition of $A_\theta(v_0)$ we obtain that $c_{v_2}(-\infty) = v_2^{-} \in \Pa_\theta(v_0)$ and $c_{w_2}(\infty) = w_2^{+} \in \Fu_\theta(v_0)$.
	
	We show that for all $0\leq \theta \leq \theta_0 := \frac{\pi}{8} \frac{\kappa_1}{\kappa_2+ 2\kappa_1}$, all $t\geq t_0>0$ for a $t_0$ depending only on $\kappa_1, \kappa_2$,
	and all $\gamma \in  \Gamma_\theta(v_0, t)$ we have
	$$
	\gamma \Fu_{\rho(\theta)}(v_0)\subset \Fu_{\rho(\theta)}(v_0),
	$$
	where $\rho(\theta)$ is defined as in the statement of the proposition.
	Consider $\eta \in \Fu_{\rho(\theta)}(v_0)$, i.e  $\sphericalangle_{p_0}( v_0^+ , \eta) \le \rho(\theta) $. The triangle inequality yields
	\begin{align*}
		\sphericalangle_{p_0}( v_0^+ , \gamma \eta)  &  \le   \sphericalangle_{p_0}( v_0^+ , w_2^+) + \sphericalangle_{p_0}( w_2^+ , \gamma \eta)\\
		&  \le \theta + \sphericalangle_{p_0}( w_2^+ , \gamma \eta).
	\end{align*}
	Denote by $x_i = \pi v_i$ and $y_i = \pi w_i$.  Since $a( \theta) \le \frac{2}{\kappa_1} \theta$ for $\theta \le \theta_0 \le \frac{\pi}{4}$,
	we obtain
	\begin{align*}
		\sphericalangle_{p_0}( w_2^+ , \gamma \eta)  &\le  \kappa_2 d(p_0, y_2) + \sphericalangle_{y_2}( w_2^+ , \gamma \eta) \\
		&\le  \kappa_2 d(p_0, y_2)  + \sphericalangle_{y_2}( w_2^+ , \gamma x_2) + \sphericalangle_{y_2}( \gamma x_2,  \gamma \eta) \\
		&\leq  \kappa_2 a(\theta) +\sphericalangle_{y_2}(\gamma y_1, \gamma x_2) + \sphericalangle_{\gamma ^{-1}y_2}( x_2,  \eta)\\
		&\leq 2 \frac{\kappa_2}{\kappa_1} \theta +\sphericalangle_{y_2}(\gamma y_1, \gamma x_2) + \sphericalangle_{\gamma ^{-1}y_2}( x_2,  \eta).
	\end{align*}
	Since $d( y_2, \gamma y_1) =t$ and $d( y_2, \gamma x_2) \ge d(   y_2, \gamma y_1)- d (\gamma y_1, \gamma x_2) \ge t -2a(\theta) $,
	the estimate \eqref{eq:angle-bound} in Lemma \ref{lem:tri2}
	yields a constant $b$ depending on $\kappa_1$ such that
	\begin{equation} \label{eqn1:prop}
		\begin{split}
		\sphericalangle_{y_2}(\gamma y_1, \gamma x_2)  &  \le b e^{- \kappa_1(t -  d ( y_1, x_2 ))}a(\theta) \\
		& \le b e^{- \kappa_1(t -  2 a(\theta))}a(\theta)  \le \frac{ 2b}{\kappa_1}  e^{- \kappa_1(t -  2 a(\theta))}\theta.
		\end{split}
	\end{equation}
	In particular this implies that we can choose $t_1 \ge 0$ only depending on $\kappa_1, \kappa_2$ 
	such that $\sphericalangle_{y_2}(\gamma y_1, \gamma x_2) \le \frac{\theta}{2} $ for all $t \ge t_1$.
	Now we want to estimate $\sphericalangle_{\gamma ^{-1}y_2}( x_2,  \eta) $ from above. Note that
	\begin{align*}
		\sphericalangle_{x_2}( \eta, v_2^+ )  &\le \kappa_2d(p_0, x_2) +\sphericalangle_{p_0}( v_2^+ ,\eta) \\
		&\le \kappa_2 d(p_0, x_2) +\sphericalangle_{p_0}( v_2^+ ,v_0^+) +\sphericalangle_{p_0}( v_0^+ ,\eta) \\
		&\le \kappa_2 a(\theta) + \theta + 2(\frac{\kappa_2}{\kappa_1}+1)\theta \\
		&\le 4\frac{\kappa_2}{\kappa_1}+3\theta.
	\end{align*}
	From the triangle inequality we obtain
	$$
	\pi = \sphericalangle_{x_2}( \gamma^{-1} x_1, v_2^+)  \le\sphericalangle_{x_2}(\gamma^{-1} x_1 , \gamma^{-1} y_2 ) +  \sphericalangle_{x_2}( \gamma^{-1} y_2 , \eta)
	+ \sphericalangle_{x_2} (\eta, v_2^+),
	$$
	and therefore
	\begin{align*}
		\sphericalangle_{x_2}( \eta, \gamma^{-1} y_2 )& \ge \pi - \sphericalangle_{x_2}( \eta, v_2^+ ) - \sphericalangle_{x_2}(\gamma^{-1} x_1 , \gamma^{-1} y_2 )\\
		& \ge \pi -  4\frac{\kappa_2}{\kappa_1} -3\theta    - b' e^{- \kappa_1(t -  2 a(\theta))} \frac{2}{\kappa_1} \theta,
	\end{align*}
	where the estimate $\sphericalangle_{x_2}(\gamma^{-1} x_1 , \gamma^{-1} y_2 )\leq b' e^{- \kappa_1(t -  2 a(\theta))} \frac{2}{\kappa_1} \theta$ follows as in \eqref{eqn1:prop}
	and $b'>0$ is a constant only depending on $\kappa_1$.
	In particular, one can choose $t_2 \ge t_1$ such that
	$$
	\sphericalangle_{x_2}( \eta, \gamma^{-1} y_2 ) \ge \pi -  4(\frac{\kappa_2}{\kappa_1}+1)\theta
	$$
	for all  $t \ge t_2$.  Let  $z$ be the orthogonal projection of $x_2$ onto the geodesic ray connecting
	$\gamma ^{-1}y_2 $ and $\eta$.	
	Using Corollary \ref{cor:tri} we obtain
	$$
	d(z, x_2) \le a( 2(\frac{\kappa_2}{\kappa_1}+1)\theta) \le \frac{4(\kappa_2 + \kappa_1)}{\kappa_1^2} \theta
	$$
	for all  $\theta \le \frac{\pi}{8} \frac{\kappa_1}{\kappa_2+ \kappa_1}  = \theta_0 $.
	Finally,  inequality  \eqref{eq:lawsin-kappa_2} in Lemma \ref{lem:tri1}
	yields the requested upper bound
	\[
	\sphericalangle_{\gamma ^{-1}y_2}( x_2,  \eta) = \sphericalangle_{\gamma ^{-1}y_2}( x_2,  z) \le
	\arcsin (\frac{\sinh(   \frac{4(\kappa_2 +\kappa_1)}{\kappa_1^2} \theta)}{\sinh(\kappa_1 d(\gamma ^{-1}y_2, z))})
	\]
	of $\sphericalangle_{\gamma ^{-1}y_2}( x_2,  \eta)$.
	Using
	$$
	d(\gamma ^{-1}y_2, z) \ge d(\gamma ^{-1}y_2, x_2) -d(x_2, z) \ge t -2a(\theta) - \frac{4(\kappa_2 + \kappa_1)}{\kappa_1^2} \theta
	$$
	and the estimate \eqref{eqn1:prop}
	we can choose $t_0 \ge t_2$ such that for all $t \ge t_0$ we have
	$$
	\sphericalangle_{y_2}(\gamma y_1, \gamma x_2) + \sphericalangle_{\gamma ^{-1}y_2}( x_2,  \eta)\le \theta,
	$$
	which yields
	$$
	\sphericalangle_{p_0}( w_2^+ , \gamma \eta) \le  2 \frac{\kappa_2}{\kappa_1}\theta  +\theta.
	$$
	Hence
	$$
	\sphericalangle_{p_0}( v_0^+ , \gamma \eta) \le   \sphericalangle_{p_0}( v_0^+ , w_2^+) +  \sphericalangle_{p_0}( w_2^+ , \gamma \eta) \le     2( \frac{\kappa_2}{\kappa_1}+1)\theta,
	$$
	which implies the assertion.

\end{proof}

From Proposition \ref{prop:c} follows that any $\gamma\in \Gamma_{\theta}(v_0,t)$ for $\theta$ small enough and $t$ large enough has an axis whose length is comparable with $t$
and which is not far away from a geodesic with $\gamma$- recurrent initial vector. More precisely:

\begin{proposition}\label{lemma:length-gamma}
	Let $\theta_0$, $t_0$, and $\rho(\theta)$ as in Prop. \ref{prop:c}. Then for all $\theta\leq \theta_0$ and all $t\geq t_0$
	any $\gamma\in \Gamma_{\theta}(v_0,t)$ has an axis $c$
	with $c(-\infty)\in \Pa_{\rho(\theta)}(v_0)$ and $c(+\infty)\in \Fu_{\rho(\theta)}(v_0)$ whose length satisfies
	\begin{equation}\label{eq:length-gamma}
		t -  \frac{4}{\kappa_1}\Big(\frac{2\kappa_2}{\kappa_1}+3\Big)\theta \le \vert\gamma\vert  \leq \frac{4}{\kappa_1}\theta +t.
	\end{equation}
	Moreover any geodesic $c_v\colon [0,t]\rightarrow X$ with $\dot c_v(0)=v\in SX$ such that $v\in  A_\theta(v_0) \cap \phi^{-t}\gamma_* A_\theta(v_0)$ satisfies
	\begin{equation}\label{eq:dist}
		d(c(s), c_v(s))\leq \frac{6}{\kappa_1}\Big( \frac{2\kappa_2}{\kappa_1}+3\Big)\theta \quad\forall\, s\in[0,t].
	\end{equation}
\end{proposition}

\begin{proof}
	Let $\rho(\theta)=2\big(\frac{\kappa_2}{\kappa_1}+1\big)\theta$.
	The existence of an axis $c$ of $\gamma$ with $c(-\infty)\in \Pa_{\rho(\theta)}(v_0)$ and $c(+\infty)\in \Fu_{\rho(\theta)}(v_0)$
	follows directly from Proposition \ref{prop:c} and Lemma \ref{lem:closing-0}.
	
	To prove the estimate for $\vert \gamma\vert$, we consider a vector $v\in A_\theta(v_0) \cap \phi^{-t}\gamma_* A_\theta(v_0)$.
	Hence,  $v\in A_\theta(v_0)$ and  $\phi^t(v)\in \gamma_* A_\theta(v_0)$. The definition of  $A_\theta(v_0)$ implies $(v^-,v^+)\in  \Pa_\theta(v_0)\times \Fu_\theta(v_0)$,  $d(\pi v_0,\pi v)\leq a(\theta)$ and $ d(\gamma (\pi v_0), \pi \phi^t (v)) \leq a(\theta)$.
	
	Using the triangle inequality we obtain
	\begin{align*}
		a(\theta) \geq d(\gamma (\pi v_0), \pi \phi^t(v)) &\geq d(\gamma (\pi v_0), \pi v) - d(\pi v, \pi \phi^t(v)) \\
		&\geq d(\gamma (\pi v_0), \pi v_0) - d(\pi v_0, \pi v) -t \\
		&\geq d(\gamma (\pi v_0), \pi v_0) -a(\theta)-t,
	\end{align*}
and therefore $d(\gamma (\pi v_0), \pi v_0) -t \leq 2a(\theta)$. Since $\vert \gamma \vert=\inf_{p \in X} d(p,\gamma p)$ and $a(\theta)\leq \frac{2}{\kappa_1}\theta$ for $\theta\leq\theta_0\leq \frac{\pi}{4}$, we obtain
\[
\vert \gamma \vert -t \leq 2a(\theta) \leq \frac{4}{\kappa_1}\theta.
\]

To obtain a lower bound for   $\vert\gamma\vert$  we assume that the axis $c$ is parametrized such that $d(\pi v_0, c(0)) =d(\pi v_0, c)$. We then obtain
\begin{equation}\label{eq:1}
d(\pi v_0, c(0)) =d(\pi v_0, c) \le  a(\rho(\theta)) \leq \frac{2}{\kappa_1}\rho(\theta),
\end{equation}
due to $(c(-\infty), c(+\infty))\in \Pa_{\rho(\theta)}(v_0)\times \Fu_{\rho(\theta)}(v_0)$ and Lemma \ref{lem:dis}.
Since  $\gamma c(0) =c( \vert \gamma \vert )$, the triangle inequality implies
\begin{align*}
	 \vert \gamma \vert & = d(c(0),  \gamma c(0)) \ge d(\pi v_0, \gamma \pi v_0) - d(  \gamma c(0),  \gamma(\pi v_0)) - d(\pi v_0, c(0))\\
	&  \ge d(\pi v_0, \gamma \pi v_0) -  \frac{4}{\kappa_1}\rho(\theta) \\
	&\geq d(\pi v,\pi \phi^t(v)) - d(\pi \phi^t(v), \gamma \pi v_0) -d(\pi v_0,\pi v) -\frac{4}{\kappa_1}\rho(\theta)\\
	& \geq t -2a(\theta) -\frac{4}{\kappa_1}\rho(\theta),
\end{align*}
which yields the lower bound
$$
\vert \gamma \vert =T \ge t -  \frac{4}{\kappa_1}(\theta + \rho(\theta)) = t -  \frac{4}{\kappa_1}\Big(\frac{2\kappa_2}{\kappa_1}+3\Big)\theta,
$$
as claimed.

Finally, let $c_v\colon[0,t]\rightarrow X$
be the geodesic with $\dot c_v(0)=v$.
Since $d(\pi v_0, \pi v)\leq a(\theta)$, $\dot c_v(t)=\phi^t(v)\in\gamma_* A_\theta(v_0)$ and using Eq. \eqref{eq:1}, we have
\[
d(c(0), c_v(0))\leq d(c(0), \pi v_0)+ d(\pi v_0, c_v(0))\leq  a(\rho(\theta))+a(\theta),
\]
and
\begin{align*}
d(c(\vert\gamma\vert), c_v(t)) & \leq d(c(\vert\gamma\vert), \gamma(\pi v_0)) + d(\gamma(\pi v_0), c_v(t))\\
&\leq a(\rho(\theta)) +a(\theta),
\end{align*}
where we also used that $\gamma c(0)=c(\vert\gamma\vert)$ and that $\gamma$ is an isometry.
The above estimate together with \eqref{eq:length-gamma} implies
\begin{align*}
	d(c(t), c_v(t))& \leq d(c(t), c(\vert\gamma\vert) + d(c(\vert\gamma\vert), c_v(t))\\
	& \leq \vert t-\vert\gamma\vert \vert +a(\rho(\theta))+a(\theta)\\
	& \leq 2a(\rho(\theta)) +2a(\theta) + a(\rho(\theta))+a(\theta)= 3a(\rho(\theta))+3a(\theta).
\end{align*}
By convexity of the function $s\mapsto d(c(s), c_v(s))$ and the estimates for $a(\theta)$ and $a(\rho(\theta))$ we obtain
\[
d(c(s), c_v(s))\leq 3a(\rho(\theta)) + 3a(\theta) \leq \frac{6}{\kappa_1}\Big( \frac{2\kappa_2}{\kappa_1}+3\Big)\theta,
\]
for all $s \in [0,t] $ which concludes the proof.
\end{proof}

The estimate in the above proposition can be somewhat improved if we additionally assume that the 
vector $v$ satisfies $\gamma \pi v = \pi\phi^t(v)$.

	\begin{proposition}\label{lemm:length-gamma-2}
	Let $\theta_0$, $t_0$, and $\rho(\theta)$ as in Prop. \ref{prop:c} and let
	$\gamma\in\Gamma_{\theta}(v_0,t)$ with $\theta\leq \theta_0$ and $t\geq t_0$. 
	Assume further that there exists a vector $v\in A_\theta(v_0)\cap \phi^{-t}\gamma_*A_\theta(v_0)$
	such that $\gamma \pi v= \pi \phi^t(v)$. 
	Then $\gamma$ 	has an axis $c$ with $c(-\infty)\in \Pa_{\rho(\theta)}(v_0)$ and $c(+\infty)\in \Fu_{\rho(\theta)}(v_0)$
	whose length satisfies
	\begin{equation}\label{eq:length-gamma-1}
		0 \le t- \vert\gamma\vert  \leq \frac{4}{\kappa_1}\Big(\frac{2\kappa_2}{\kappa_1}+3\Big)\theta.
	\end{equation}
	Furthermore
	\begin{equation}\label{eq:dist-1}
		d(c(s), c_{v}(s))\leq \frac{6}{\kappa_1} \Big(\frac{2\kappa_2}{\kappa_1}+3\Big)\theta \quad\forall s\in[0,t],
	\end{equation}
	where $c_{v}\colon [0,t]\rightarrow X$ is the geodesic with starting vector $v$.
\end{proposition}

\begin{proof}
	Let $\rho(\theta)=2\big(\frac{\kappa_2}{\kappa_1}+1\big)\theta$.
	As in the previous proposition, the existence of the axis $c$ such that $c(-\infty)\in \Pa_{\rho(\theta)}(v_0)$
	and $c(+\infty)\in \Fu_{\rho(\theta)}(v_0)$ follows directly  from Proposition \ref{prop:c} and Lemma \ref{lem:closing-0}.
	Clearly, the length of $c$ is $\vert\gamma\vert$.
	
	Since $\vert \gamma \vert =\inf_{p \in X} d(p,\gamma p)$ and $\gamma \pi v=\pi \phi^t(v)$,
	we have
	$$
	t=d(\pi v, \pi\phi^t(v)) = d(\pi v, \gamma\pi v)\geq  \vert \gamma \vert.
	$$
	Parametrizing $c$ such that $d(\pi v_0,c(0))= d(\pi v_0, c)$, we have 
	\[
		d(\gamma \pi v, \gamma c(0))= d(\pi v, c(0))
		\leq d(\pi v, \pi v_0) + d(\pi v_0, c(0))  
		 \leq a(\theta) + a(\rho(\theta)),
	\]
	which follows from Lemma \ref{lem:dis} and the assumption $v\in A_\theta(v_0)$. 
	
	Using the reverse triangle inequality we obtain
	\begin{align*}
		\vert \gamma \vert & = d(c(0), \gamma c(0)) \geq d(\gamma \pi v, \pi v) 
		- d(\pi v, c(0)) -d(\gamma \pi v, \gamma c(0))\\
		& \geq d(\pi \phi^t(v), \pi v) -2a(\theta)  -2a(\rho(\theta)) \geq t -\frac{4}{\kappa_1}\theta -\frac{4}{\kappa_1}\rho(\theta),
	\end{align*}
	which implies the estimate \eqref{eq:length-gamma-1} using the definition of $\rho(\theta)$ 
	and the estimate $a(\theta)\leq \frac{2}{\kappa_1}\theta$ for $\theta\leq \theta_0\leq \frac{\pi}{4}$.
	
	It remains to estimate the distance between $c(s)$ and $c_v(s)$ for $s\in[0,t]$.
	Since
	\begin{align*}
		d(c(\vert \gamma\vert ), c_{v}(t))& = d(c(\vert \gamma\vert), \pi \phi^t (v)) = d(\gamma c(0), \gamma \pi  v)\\
		& = d(c(0), \pi v)\leq a(\theta) + a(\rho(\theta)),
	\end{align*}
	we have
	\begin{align*}
		d(c(t), c_{v}(t))&\leq d(c(t), c(\vert \gamma \vert)) +d( c(\vert\gamma\vert), c_{v}(t))\\
		& \leq t-\vert \gamma \vert +a(\theta) + a(\rho(\theta))\leq 3a(\theta) + 3a(\rho(\theta)).
	\end{align*}
	By convexity of the function $s\mapsto d(c(s), c_{v}(s))$, we have
	\[
	d(c(s), c_{v}(s))\leq 3a(\theta) + 3a(\rho(\theta)) \leq \frac{6}{\kappa_1} \Big(\frac{2\kappa_2}{\kappa_1}+3\Big) \quad\forall s\in[0,t].
	\]
\end{proof}	

\begin{remark}\label{rmk:improved-length}
	If one further assumes that $\pi v=\pi v_0$, the estimates \eqref{eq:length-gamma-1} 
	and \eqref{eq:dist-1} become
	\[
	0 \le t- \vert\gamma\vert  \leq \frac{8}{\kappa_1}\Big(\frac{\kappa_2}{\kappa_1}+1\Big)\theta
	\]
	and 
	\[
	d(c(s), c_{v}(s))\leq \frac{12}{\kappa_1}\Big(\frac{\kappa_2}{\kappa_1}+1\Big)\theta \quad\forall s\in[0,t],
	\]
	respectively, since $d(\pi v, c(0))=d(\pi v_0, c(0))\leq a(\rho(\theta))$.
\end{remark}

We can finally state the announced quantitative version of the closing lemma for the geodesic flow on the manifold $(M,g)$.
For  that we recall the definition of the dynamical metrics $d_1$ on $SM$ and introduce its lift $\tilde d_1$ onto $SX$:
$$
d_1(v,w)=\max_{t\in[-1,1]} d(c_v(t), c_w(t)), \;  \text {for all}  \; v, w \in SM
$$
and
$$
\tilde d_1(v,w)=\max_{t\in[-1,1]} \tilde d(c_v(t), c_w(t)), \;  \text {for all}  \; v, w \in SX,
$$
where  $d$ is the metric on $M$ and $\tilde d$ its lift on $X$.
We also need the following Lemma.
\begin{lemma}\label{lem:theta_0}
Consider the function
$a(\theta):=  \frac{1}{\kappa_1} \arcosh \left(\frac{1}{\cos (\theta)} \right) $
introduced in Lemma \ref{lem:dis}. Then
$\beta\leq a(\kappa_1\beta)$  for all $\beta \in [0,  \frac{\pi}{2 \kappa_1}]$ and
$a$ is monotonically increasing on $[0,  \min\big\{\frac{\pi}{2},\frac{\pi}{2 \kappa_1}\big\}]$.

\end{lemma}
\begin{proof}
	The function $ \beta \mapsto  a(\kappa_1\beta)$ is non-negative on  $[0,  \frac{\pi}{2 \kappa_1}]$ and assumes the values $0$ resp. $\infty$ at $0$ resp. $\frac{\pi}{2 \kappa_1}$.
Note, that $\beta\leq a(\kappa_1 \beta)$ is equivalent to
$$
g(\beta) := \cosh(\kappa_1\beta)\cos(\kappa_1 \beta)\leq 1.
$$
Since $g(0)=1$ and  $g(\frac{\pi}{2 \kappa_1})=0$, it is enough to prove that $g'(\beta) \le 0$ for all $\beta \in [0,  \frac{\pi}{2 \kappa_1}]$. We have
 $$
 g'(\beta)=\kappa_1 \sinh(\kappa_1 \beta)\cos(\kappa_1\beta)-\kappa_1\cosh(\kappa_1\beta)\sin(\kappa_1\beta),
 $$
 which in particular  implies $g'(0) =0$ and that $g'(\frac{\pi}{2\kappa_1})<0$. Since
 $$
 g''(\beta) = -2 \kappa_1^2  \sinh(\kappa_1 \beta)\sin(\kappa_1\beta)) \le 0
$$
we obtain $g'(\beta) \le 0$ on $[0,  \frac{\pi}{2 \kappa_1}]$, which proves the first part of the lemma.

Finally, for $\theta\in [0,\pi/2)$ the function $a(\theta)$ is monotonically increasing, hence also on the interval $\big[0,  \min\big\{\frac{\pi}{2},\frac{\pi}{2 \kappa_1}\big\}]$.

\end{proof}

\begin{theorem}\label{thm:closing}
There exist $\delta_0=\delta_0(\kappa_1,\kappa_2)\in\big(0,\frac{1}{2}\big)$ and $t_0=t_0(\kappa_1,\kappa_2)>1$
such that for all $\delta\leq \delta_0$, all $T\geq t_0$ and all $w\in SM$ with $d_1(w,\phi^T(w))\leq \delta$
there exist $u\in SM$, $T'>0$, and $C= \frac{2}{\kappa_1}\Big(\frac{2\kappa_2}{\kappa_1}+3\Big)\max\{2\pi,\kappa_1\}$ such that $\phi^{T'}(u)=u$,
\[
\vert T-T' \vert \leq 2 C \delta,
\]
and
\[
d_1(\phi^s(w), \phi^s(u))\leq (5C +1)\delta \quad\forall\, s\in[0,T].
\]
\end{theorem}

\begin{proof}
Let $c_w\colon [0,T]\rightarrow M$ be the geodesic such that $\phi^t(w)=\dot c_w(t)$ for all $t\in[0,T]$.

Consider the lift $\tilde c$ of $c_w$ on the universal cover $X$  of $M$, let $p=\tilde{c}(0)$ and $v_0= \dot {\tilde c}(0)$.

From the assumption $d_1(w,\phi^T(w))\leq\delta$ follows the existence of $\gamma\in\Gamma\setminus \id$ such that
$\tilde d_1(\gamma_* v_0,  \dot {\tilde c}(T))\leq\delta$, i.e. $\tilde d(\gamma(\tilde c(s)), \tilde c(T+s)))\leq\delta$, for $s \in [-1,1]$.
We first observe that
\begin{align*}
	\delta & \geq \tilde d_1(\gamma_*(v_0), \dot {\tilde c}(T))=\tilde d_1 (v_0, \gamma^{-1}_*\phi^T(v_0))=
	 \tilde d_1 (v_0, \phi^{T}\gamma^{-1}_*(v_0)).
\end{align*}
From Lemma \ref{lemma:angles-at-infinity} and Remark \ref{rmk:angle-at-infinity-approx}
 follows
\begin{equation}\label{eq:auxiliary}
(\gamma^{-1}_*\phi^T(v_0))^+ \in \Fu_{2 \pi \delta}(v_0)
 \; \text{and} \; (\gamma^{-1}_*\phi^T(v_0))^- \in \Pa_{2 \pi \delta}(v_0)
\end{equation}
for all $ \delta \in (0, \frac{1}{2}) $.

 Let now $\theta_0$ and $t_0$ as in Proposition \ref{prop:c},
 and set $\theta_\delta:=\max\{2\pi,\kappa_1\}\delta$.
	Choose $\delta_0\in(0, \frac{\pi}{2 \kappa_1}]$ such that $\theta_\delta\leq\theta_0$ for all $\delta\in(0,\delta_0]$, and
	let $T\geq t_0$.
 Recall that
 \begin{multline}\label{theta}
 	A_\theta(v_0) := \{v\in SX \mid (v^-, v^+) \in \Pa_\theta(v_0)\times\Fu_\theta(v_0) \\
 	\text{ and } \tilde d(\pi v, \pi v_0) \le a(\theta) \}.
 \end{multline}
We claim that
$$
A_{\theta_\delta}(v_0) \cap \phi^{-T}\gamma_* A_{\theta_\delta}(v_0) \not= \emptyset.
$$
Since $v_0 \in A_{\theta_\delta}(v_0) $ it is enough to show that $\phi^{-T}\gamma_*(v_0) \in A_{\theta_\delta}(v_0)$.
From \eqref{eq:auxiliary}  and the definition of $\theta_\delta$ it follows
$$
 ( (\gamma^{-1}_*\phi^T(v_0))^-, (\gamma^{-1}_*\phi^T(v_0))^+ ) \in  \Pa_{ \theta_\delta}(v_0) \times  \Fu_{ \theta_\delta}(v_0).
$$
Since
\[
\tilde d(\pi (v_0),\pi \phi^{-T}\gamma_*(v_0)) \leq \tilde d_1 (v_0, \phi^{T}\gamma^{-1}_*(v_0)) \leq \delta \le a(\kappa_1 \delta)  \leq a(\theta_\delta),
\]
 by Lemma \ref{lem:theta_0} and by the choice of $\theta_\delta$, the claim is proved.


Proposition \ref{lemma:length-gamma} now implies the existence of an axis $c$ of $\gamma$ of length $T'$ such that
\begin{equation}\label{eq:action-difference-closing}
	\begin{split}
	\vert T- T' \vert & \leq \frac{4}{\kappa_1}\Big(\frac{2\kappa_2}{\kappa_1}+3\Big)\theta_\delta \\
	&\leq \frac{4}{\kappa_1}\Big(\frac{2\kappa_2}{\kappa_1}+3\Big)\max\{2\pi,\kappa_1\}\delta =: 2C\delta
		\end{split}
\end{equation}
and
\[
\tilde d(\tilde c(t), c(t)) \leq  \frac{6}{\kappa_1}\Big( \frac{2\kappa_2}{\kappa_1}+3\Big)\theta_\delta=3C\delta\qquad \forall\, t\in[0,T].
\]
Since $d_1(w,\phi^T(w))\leq \delta$ we have for the lift $\tilde c$ of $c_w$ that
$$
 \tilde d( \gamma \tilde c(s) , \tilde c(s+T)) \le \delta
$$
for $ s \in [-1,1]$.
This implies for $s\in[0,1]$
\begin{align*}
	\tilde d(\tilde c(T+s), c(T+s)) & \leq \tilde d( \tilde c(T+s), c(T'+s))+ \tilde d (c(T'+s), c(T+s))\\
	& = \tilde d(\tilde c(T+s), \gamma  c(s)) + \vert T- T'\vert\\
	& \leq  \tilde d(\tilde c(T+s), \gamma  \tilde c(s)) +\tilde d(\gamma \tilde c(s), \gamma  c(s)) + \vert T- T'\vert\\
	& \leq \delta +\tilde d(\tilde c(s), c(s)) +2C\delta\\
	& \leq \delta + 3C\delta+2C\delta = (5C+1)\delta.
\end{align*}
For  $ s \in [-1,0]$ we calculate
\begin{align*}
	\tilde d(\tilde c(s), c(s)) & = \tilde d( \gamma \tilde c(s),  \gamma c(s)) \\
                                & \leq  \tilde d(\gamma \tilde c(s), \tilde c(s+T)) +\tilde d(\tilde c(s+T), c(s+T)) +\tilde d(c(s+T), c(s+T'))\\
                                & \leq \delta + 3C\delta +\vert T- T'\vert \leq (5C+1)\delta,
\end{align*}
since $\gamma c(s)=\gamma(s+T)$ and $s+T\in[T-1,T]\subset [0,T]$.
Putting all the estimates together yields
$$
\tilde d(\tilde c(t), c(t)) \le  (5C+1)\delta
$$
for all $t \in [-1,T+1]$ and therefore,
\begin{align*}
\tilde d_{1}( \dot{\tilde c}(t), \dot c(t)) &  \leq (5C +1)\delta
\end{align*}
for all $t\in[0,T]$ where $C$ is the one from \eqref{eq:action-difference-closing}.
By projecting onto $M$ we  obtain the closed geodesic $c_{u}=\pr\circ c$
with starting vector $u\in SM$ of period $T'$ satisfying \eqref{eq:action-difference-closing}.
In particular $u$ is such that $u=\phi^{T'}(u)$ and
\[
d_1(\phi^t(u), \phi^t(w))\leq (5 C +1)\delta
\]
for all $t\in[0,T]$ as desired.
\end{proof}

If $w$ and $\phi^T(w)$ have the same foot-point projection, the underlying closed geodesic is
a geodesic loop and assuming $d_1(w,\phi^T(w))$ to be small enough, it is easy to see that the deck transformation associated to (the lift of)
the geodesic loop has an axis giving the existence of a closed geodesic nearby.
In this case, Proposition \ref{lemm:length-gamma-2} gives slightly improved action difference and distance estimate.

\begin{theorem}\label{thm:closing-2}
There exist $\delta_0=\delta_0(\kappa_1,\kappa_2)\in\big(0,\frac{1}{2}\big)$ and $t_0=t_0(\kappa_1,\kappa_2)>1$ such that for all $\delta\leq \delta_0$, all $T\geq t_0$ and all $w\in SM$ such that $d_1(w,\phi^T(w))\leq \delta$
and $\pi w=\pi\phi^T(w)$ there exist $u\in SM$, $T'>0$, and $\tilde C = \frac{4\pi}{\kappa_1}\big(\frac{\kappa_2}{\kappa_1}+1\big)$ such that $\phi^{T'}(u)=u$,
\[
0 < T-T' \leq 4 \tilde C \delta,
\]
and

\[
d_1(\phi^s(w), \phi^s(u))\leq (10 \tilde C +1)\delta \quad\forall\, s\in[0,T].
\]

\end{theorem}
\begin{proof}
	Using the same notation as in Theorem \ref{thm:closing}, the assumptions imply the existence of $\gamma\in\Gamma\setminus \id$
	such that $\gamma p= \tilde c(T)$ and $\tilde d_{1}(\gamma_* v_0, \dot {\tilde c}(T))\leq \delta$.
	
	Let $t_0$ and $\theta_0$ be as in Proposition \ref{prop:c} and choose $\delta_0= \frac{1}{2\pi}\theta_0$.
	Further, let $T\geq t_0$ and set $\theta_\delta=2\pi\delta$ for all $\delta\in[0,\delta_0]$. Then $\theta_\delta\leq\theta_0$ for $\delta\leq \delta_0$.
	
	As before, we prove that $\gamma_*^{-1}\phi^T(v_0)\in A_{\theta_\delta}(v_0)$.
	We have
	$$
	0 = \tilde d(\gamma p, \tilde c(T))=  \tilde d (p,\gamma^{-1}\tilde c(T))=  \tilde d(p, \pi \gamma_*^{-1}\phi^{T} v_0),
	$$
	and
	\[
		(\gamma^{-1}_*\phi^T(v_0))^+ \in \Fu_{\theta_\delta}(v_0)
		\; \text{and} \; (\gamma^{-1}_*\phi^T(v_0))^- \in \Pa_{\theta_\delta}(v_0)
	\]
	for all $ \delta \in (0, \frac{1}{2}) $ by Lemma \ref{lemma:angles-at-infinity} and Remark \ref{rmk:angle-at-infinity-approx}, which proves the claim.
	
	By Proposition \ref{lemm:length-gamma-2} and Remark \ref{rmk:improved-length} there exists an axis $c$ of $\gamma$ of length $T'<T$ satisfying
	\[
	0 \le T- T'  \leq \frac{8}{\kappa_1}\Big(\frac{\kappa_2}{\kappa_1}+1\Big)\theta_\delta=\frac{16 \pi}{\kappa_1}\Big(\frac{\kappa_2}{\kappa_1}+1\Big)\delta =: 4 \tilde C \delta ,
	\]
	and such that
	\[
	\tilde d(c(t), \tilde c(t))\leq \frac{12}{\kappa_1}\Big(\frac{\kappa_2}{\kappa_1}+1\Big)\theta_\delta=6\tilde C \delta  \qquad\forall t\in[0,T].
	\]
	For $s\in[0,1]$ we follow the estimate of the previous proposition to obtain
	\begin{align*}
	\tilde d (c(T+s), \tilde c(T+s))& \leq \tilde d( c(T+s), \gamma c(s)) + \tilde d(\gamma c(s), \gamma \tilde c(s) ) +\tilde d( \gamma \tilde c(s), \tilde c( T+s))\\
	& \leq T- T' + 6 \tilde C \delta +\delta = (10\tilde C +1)\delta.
	\end{align*}
	Similarly, for $s\in[-1,0]$ we obtain $\tilde d (c(s), \tilde c(s))\leq (10\tilde C +1)\delta$.
	Projecting onto $M$ we obtain the desired claim.
\end{proof}

\section{Construction of partner orbits }\label{sec:partner-orbits-general}

	Using the results of Section \ref{sec:closing} we can now prove the existence of partner orbits of self-crossing geodesics, together with the fact that they have smaller period
	and that they remain close to the original geodesic.
	
	We recall that a closed geodesic $c_w$ has a self-crossing at $p=c_w(0)$ with crossing angle $\eps$
	if there is a time $T_1\in(0,T)$ such that
	$c_w(0) =c_w(T_1)  $ and $\sphericalangle_{\pi w}( w, -\dot{c_w}(T_1))=\eps$.
	We will also refer to such geodesics as closed geodesic with self-crossing at $p$ at time $T_1$ to emphasise the time at which the
	intersection takes place.
	

\begin{theorem}\label{thm:partner-general}
	Let $M$ be a compact Riemannian manifold with sectional curvature $-\kappa_2^2 \leq K \leq -\kappa_1^2$ for $0<\kappa_1 \leq \kappa_2$.
	Let $\eps_0:= \frac{\pi}{16}\frac{\kappa_1}{\kappa_1+\kappa_2}$ and let $T_0=t_0(\kappa_1,\kappa_2)>1$ be the one from Prop. \ref{prop:c}.
	Then for all $\eps\leq \eps_0$, all $T_1, T_2\geq T_0$ and all closed geodesic $c_w: [0, T] \to M$ of period $T=T_1+T_2$
	with a self-crossing at $p=c_w(0)$ at time $T_1$ and crossing angle $\eps$ there exists a closed orbit $c_u:  [0, T'] \to M$ of period $T'<T$ satisfying
	\[
	T-T'  \leq \Big(\frac{18}{\kappa_1} + \frac{16\kappa_2}{\kappa_1^2}\Big) \eps,
	\]
	and such that for some $T_3\in[0,T^\prime]$ we have
	$$
	d(  c_u(s),  c_w[0, T_1]) \le   \left( \frac{25}{\kappa_1} +\frac{24\kappa_2}{\kappa_1^2} \right) \eps
	$$
    for $s \in [ 0,T_3]$  and
$$
	d(  c_u(s),  c_{-w}[0, T_2]) \le  \left( \frac{25}{\kappa_1} +\frac{24\kappa_2}{\kappa_1^2} \right) \eps,
	$$
	for $s \in [ T_3, T^\prime]$. Summarizing both inequalities yields
	$$
	d(  c_u(s), c_w[0, T] )\le \left( \frac{25}{\kappa_1} +\frac{24\kappa_2}{\kappa_1^2} \right) \eps
	$$
	for all $s \in [0, T^\prime]$.
\end{theorem}	


\begin{proof}
	 Let $c_w:[0,T]  \to M  $ be as in the statement of the theorem
	  and denote by $c_{w,1}\colon[0,T_1]\to M$, $c_{w,1}(t)=c_w(t)$
	and $c_{w,2}\colon[0,T_2]\rightarrow M$, $c_{w,2}(t)=c_w(T_1+t)$, where $T_2:=T-T_1$ the associated geodesic loops.	
	Consider the geodesic loop $\bar c_{w,2}: [0, T_2] \to M $   given by $\bar c_{w,2}(t)= c_{w}(T-t)$ reversing the orientation of $c_{w,2}$.
	Let $\tilde c_{w,1}$ and $\tilde c_{w,2}$ be the lifts of  $ c_{w,1}$ and   $c_{w,2}$ to the universal cover $X$ such that $\tilde c_{w,1}(T_1) = \tilde c_{w,2}(0) =:p_1$ and let
	$\gamma_1$ and $\gamma_2$ be the deck transformations 
	with $\gamma_1(\tilde c_{w,1}(0)) =\tilde c_{w,1}(T_1) $ and  $\gamma_2(\tilde c_{w,2}(0)) = \tilde c_{w,2}(T_2)$. Then $\tilde \alpha: [0, T_2] \to X$ with $\tilde \alpha(t) = \tilde c_{w,2}(T_2-t) $ is a lift  of $\bar c_{w,2}$
	with initial point $\tilde \alpha (0) = \tilde c_{w,2}(T_2) = \gamma_2 p_1$ and therefore
	$\tilde c_2(t) = \gamma_2^{-1}\tilde \alpha(t)$ is a lift of $\bar c_{w,2}$ such that $\tilde c_2(0) = p_1$. Set $\tilde c_1 =\tilde c_{w,1}$,
	$p_0:=\tilde c_1(0)$, $v_0:= \dot{\tilde c}_1(0)$ and $p_2:=\tilde c_2(T_2)$.
	Then $p_2 = \gamma_2^{-1}\tilde \alpha(T_2) = \gamma_2^{-1}(p_1)$ and  $\sphericalangle_{p_1}( \dot{\tilde c}_1(T_1), \dot{\tilde c}_2(0) ) = \eps$. Furthermore,
	$$
	\sphericalangle_{p_2}((\gamma_2^{-1}\gamma_1)_\ast v_0, \dot{\tilde c}_2(T_2))
	= \sphericalangle_{\pi w}( w, -\dot{c}_{w,2}(0)) = \eps.
	$$

	Finally, let $\tilde c_3: [0,\hat T] \to X$ be the geodesic connecting $p_0$ and $p_2$, where $\hat T=d(p_0,p_2)$.
	In particular, the points $p_0,p_1,p_2$ form a geodesic triangle in $X$. We observe that Eq. \eqref{eq:distance2} 
	in Corollary \ref{cor:tri} together with the triangle inequality yields
	\begin{equation}\label{eq:T- hat T}
	 0 \leq T- \hat T  \leq 2a\big(\frac{\eps}{2}\big) \leq \frac{2}{\kappa_1}\eps,
	\end{equation}
	where the last inequality is due to the fact that $\eps\leq \eps_0 < \pi/4$.
	Moreover, $\hat T\geq T_i \geq T_0$ for $i=1,2$. Indeed if $\hat T< T_1$, then by Corollary \ref{cor:tri} we would obtain
	$T_1+T_2 -2a(\eps/2) \leq \hat T < T_1$, leading to $T_2 < 2a(\eps/2)$,
	which contradicts Proposition \ref{prop:period-loops1}. A similar contradiction holds assuming $\hat T < T_2$.
	
	Set $\theta_1:=\sphericalangle_{p_0}( v_0, \dot{\tilde c}_3(0))$ and $\theta_2:=\sphericalangle_{p_2}( \dot{\tilde c}_3(\hat T), \dot{\tilde c}_2(T_2))$. Clearly,  $\theta_i \leq \eps$ for $i=1,2$ and setting
	$v:= \dot{\tilde c}_3(0)$ we have 
	 $v \in A_{\theta_1}(v_0)\subset A_{2\eps}(v_0)$, since 	
	 $$
	 (v^-,v^+) =  (\tilde c_3(-\infty),   \tilde c_3( + \infty)) \in ( \Pa_{\theta_1}(v_0), \Fu_{\theta_1}(v_0)).
	 $$ 
	 We also have 
	\begin{equation}\label{eq:foot-point}
	 \pi \phi^{\hat T}(v)= \tilde c_3(\hat T) = p_2=\gamma_2^{-1}\gamma_1 p_0 =\gamma_2^{-1}\gamma_1 \pi v_0 = \gamma_2^{-1}\gamma_1 \pi v,
	\end{equation}
	and
\begin{align*}
	\sphericalangle_{p_0}( v_0, (\gamma_2^{-1}\gamma_1)^{-1}_* \phi^{\hat T}(v) )& =	
	\sphericalangle_{\gamma_2^{-1}\gamma_1 p_0}( (\gamma_2^{-1}\gamma_1)_* v_0, \phi^{\hat T}(v))
	 =\sphericalangle_{p_2}((\gamma_2^{-1}\gamma_1)_* v_0, \dot{\tilde c}_3(\hat T))\\
		& \leq \sphericalangle_{p_2}( (\gamma_2^{-1}\gamma_1)_* v_0, \dot{\tilde c}_2(T_2)) + \sphericalangle_{p_2}(\dot{\tilde c}_2(T_2), \dot{\tilde c}_3(\hat T))\\
		& \leq \eps +\theta_2 \leq 2\eps.
	\end{align*}
	This implies that  $(\gamma_2^{-1}\gamma_1)^{-1}_* \phi^{\hat T}(v) \in  A_{2 \eps}(v_0)$ and therefore 
	$v \in   \phi^{-\hat T}(\gamma_2^{-1}\gamma_1)_* A_{2 \eps}(v_0)$.
	In particular $$ v \in A_{2 \eps}(v_0) \cap  \phi^{-\hat T}(\gamma_2^{-1}\gamma_1)_* A_{2 \eps}(v_0), $$ which implies
	$\gamma_2^{-1}\gamma_1\in \Gamma_{2\eps}(v_0,\hat T)$.
	Since the vector $v$ satisfies the assumptions of Proposition \ref{lemm:length-gamma-2} and Remark \ref{rmk:improved-length}, $\hat T \geq T_0$ and $\eps\leq \eps_0$, 
	there exists an axis $c$ of $\gamma_2^{-1}\gamma_1$
	of length
	 $T':=\vert \gamma_2^{-1}\gamma_1\vert $ satisfying
	\begin{equation}\label{eq:length}
		 0 \leq  \hat T- T' \leq \frac{16}{\kappa_1}\Big( \frac{\kappa_2}{\kappa_1}+1\Big)\eps,
	\end{equation}	
	and
	\begin{equation}\label{eq:distance}
		d(c(s), \tilde c_3(s)) \leq \frac{24}{\kappa_1}\Big( \frac{\kappa_2}{\kappa_1}+1\Big)\eps \qquad \forall\, s\in[0,\hat T].
	\end{equation}
	Moreover, estimates \eqref{eq:length} and \eqref{eq:T- hat T} yield
	\begin{equation}\label{eq:length2}
		\begin{split}
		0\leq T- T'&  = (T-\hat T) +(\hat T - T') \\
		& \leq \frac{2}{\kappa_1}\eps + \frac{16}{\kappa_1}\Big( \frac{\kappa_2}{\kappa_1}+1\Big)\eps= \Big(\frac{18}{\kappa_1} + \frac{16\kappa_2}{\kappa_1^2}\Big)\eps.
		\end{split}
	\end{equation}
		Let now $p_3$ be the  orthogonal  projection of $p_1$ onto $\tilde c_3$ and  $T_3 \in  [0, \hat T] $ such that
	 $\tilde c_3(T_3) = p_3$.
	 Since
	$\sphericalangle_{p_1}(p_0, p_2) \in [\pi -\eps, \pi ] $ and $\eps \in [0, \pi/2)$,  
	 we obtain from Corollary \ref{cor:tri} that $d(p_1, p_3) \le a( \frac{\eps}{2}) $. 
	 Then
	 $ d(\tilde c_1(0, T_1], \tilde c_3(s))  \le a( \frac{\eps}{2})$ for all $s \in  [0, T_3]$  
	 and  $ d(\tilde c_2(0, T_2], \tilde c_3(s))  \le a( \frac{\eps}{2})$ for $s \in  [T_3, \hat T]$.  
	 We then obtain
	 \begin{align*}
	 d( c(s),  \tilde c_1[0, T_1]) &\le d(c(s), \tilde c_3(s)) + d( \tilde c_3(s), \tilde c_1(0, T_1]) \\
	   & \le  \frac{24}{\kappa_1}\Big( \frac{\kappa_2}{\kappa_1}+1\Big)\eps + a( \frac{\eps}{2}) \le \frac{24}{\kappa_1}\Big( \frac{\kappa_2}{\kappa_1}+1\Big)\eps +\frac{\eps}{\kappa_1}\\
	   &=  \left( \frac{25}{\kappa_1} +\frac{24\kappa_2}{\kappa_1^2} \right) \eps
	\end{align*}
	 for all $s\in[0, T_3]$. Similarly
	  \begin{align*}
	 d( c(s),  \tilde c_2[0, T_2]) &\le d(c(s), \tilde c_3(s)) + d( \tilde c_3(s), \tilde c_2(0, T_2])  \\
	& \le \left( \frac{25}{\kappa_1} +\frac{24\kappa_2}{\kappa_1^2} \right) \eps
	\end{align*}
	 for all $s\in[T_3, \hat T]$.
	 Finally, projecting $s \mapsto c(s)$ onto $M$ we obtain a closed geodesic $c_u(s) =\pr \circ c(s)$ on $M$ of period
	$$
	T^\prime \in \Big(T - \Big(\frac{18}{\kappa_1} + \frac{16\kappa_2}{\kappa_1^2}\Big)\eps,T \Big).
	$$
	Using the estimates above and the fact that $T^\prime \le \hat T$ we obtain 
	$$
	d(  c_u(s),  c_w[0, T_1]) \le   \left( \frac{25}{\kappa_1} +\frac{24\kappa_2}{\kappa_1^2} \right) \eps \qquad\forall s \in [ 0,T_3]
	$$
     and
$$
	d(  c_u(s),  c_{-w}[0, T_2]) \le   \left( \frac{25}{\kappa_1} +\frac{24\kappa_2}{\kappa_1^2} \right) \eps \qquad\forall s\in[T_3, T^\prime].
	$$
	 Summarizing both inequalities yields
	$$
	d(  c_u(s), c_w[0, T]) \le  \left( \frac{25}{\kappa_1} +\frac{24\kappa_2}{\kappa_1^2} \right) \eps \qquad \forall s\in [0, T^\prime].
$$

\end{proof}

\begin{figure}[htb]\label{fig:general-construction}
	\scalebox{.7}{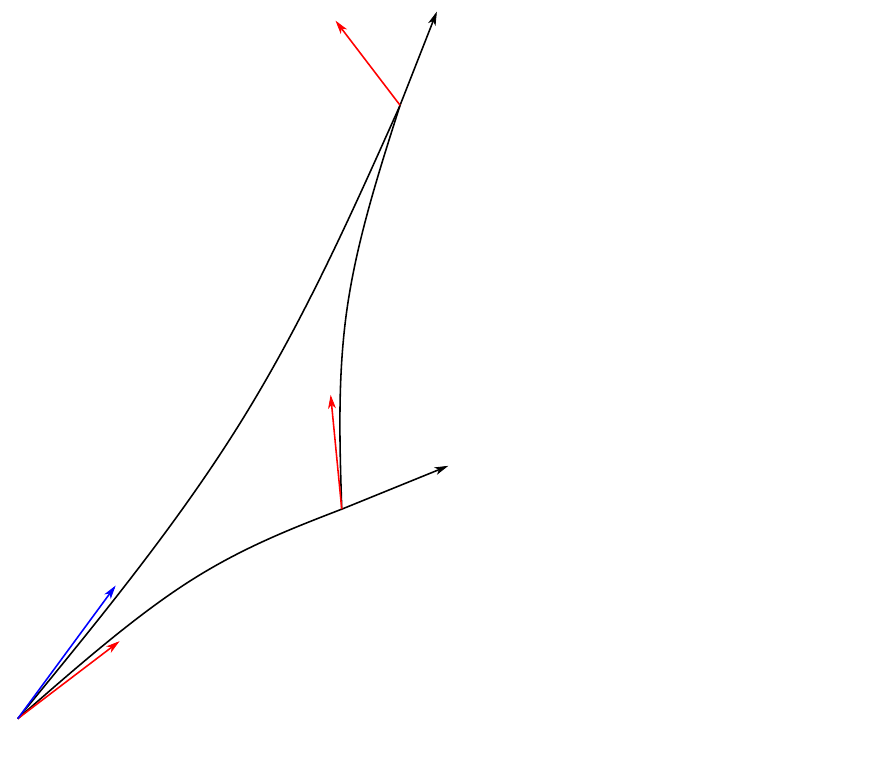}
	\caption{Proving Theorem \ref{thm:partner-general}.}
\end{figure}

\section{Construction of partner orbits for surfaces}\label{sec:partner-orbits-2dim}

In the case that the manifold $M$ has dimension $2$,
it is possible to improve the estimate for the action difference
in Theorem \ref{thm:partner-general} by means of Lemma \ref{lem:perp-proj}. 
This new estimate is consistent with the result in \cite{HK15}.
To this end, we need an estimate from above for the length of closed geodesic loops in $M$,
and one from above and below for the difference between the lengths of sides in a geodesic triangle.


\begin{proposition}\label{prop:period-loops1}
	Let $M$ be a compact $d$-dimensional manifold, $d\geq 2$, with sectional curvature $- \kappa \le K\le 0$, for some $ \kappa>0$.
	Assume that the injectivity radius $\inj(M) = \frac{\rho}{2} $ is positive.
	Let $c: [0,r]\rightarrow M$ be  a geodesic loop, i.e. $c(0) = c(r) $ and  assume $\sphericalangle(\dot c(0), - \dot c(r)) \le \eps >0$.
	Then
	\begin{equation}\label{eq:period-loops1}
		r \ge  \frac{1}{\kappa}  \arcosh \left(\frac {b}{ 1-\cos (\eps)}  +1 \right)
	\end{equation}
	with  $b= 2 (\cosh (\kappa \rho) -1) $.
\end{proposition}

\begin{proof}
	Let  $\tilde c: [0,r] \rightarrow X$ be a lift of $c$ and let  $\dot {\tilde c}(0) = v \in S_pX$.
	Then there exits $\gamma \in \Gamma$ with $\gamma (p) = \tilde c(r)$. Furthermore,
	$\sphericalangle(\gamma_*(v), - \dot {\tilde c}(r)) \le \eps $ and $\gamma \tilde c(t)$ is the geodesic with initial condition $\gamma_*(v)$.
	Since by the assumption on the injectivity radius we have
	$ d(  \tilde c(\frac{r}{2}),  \gamma (\tilde c(\frac{r}{2}) ) \ge \rho$. Using \eqref{eq:lawcosine-kappa2} in Lemma \ref{lem:tri1} we obtain
	\begin{align*}
		\cosh ( \kappa \rho ) &\le  \cosh ( \kappa  d(  \tilde c(\frac{r}{2}),  \gamma (\tilde c(\frac{r}{2}) )\le  \cosh^2 (\kappa \frac{r}{2}) - \sinh^2 (\kappa \frac{r}{2}) \cos( \eps)\\
		&= 1 +  \sinh^2 (\kappa \frac{r}{2})(1- \cos (\eps)) \\
		&= 1 + \frac{1}{2} (\cosh (\kappa r) -1)(1- \cos (\eps)) \\
	\end{align*}
	which yields
	$$
	\frac {2}{ 1- \cos (\eps)} (\cosh (\kappa \rho) -1) \le \cosh (\kappa r) -1
	$$
	and therefore
	$$
	\frac{1}{\kappa}  \left(\arcosh \left(\frac {b}{ 1-\cos (\eps)}  +1 \right)\right) \le r
	$$
	where $b= 2 (\cosh (\kappa \rho) -1) $, as claimed.
\end{proof}

\begin{lemma}\label{lem:T'-bound}
	Let $X$ be the universal cover of a compact manifold $M=X/\Gamma$ of dimension $d\geq 2$ with sectional curvature $-\kappa_2^2 \leq K \leq -\kappa_1^2$ with $0 < \kappa_1 \leq \kappa_2$.
	Let $\Delta(p_1,p_2,p_3)$ be a geodesic triangle in $X$ such that $\sphericalangle_{p_3}(p_1,p_2)=\pi-\eps$ for some $\eps\in (0,\pi)$ and such that $T_i=d(p_3,p_i)$ satisfies \eqref{eq:period-loops1}
	with $T_i$ instead of $r$ for $i=1,2$. Let $T'=d(p_1,p_2)$ and $T=T_1+T_2$, then
	\begin{multline*}
	-\frac{1}{\kappa_2}\log\Big( 1-\frac{1}{\pi^2}\Big(1-\frac{8(b+2)^2}{(b+2)^4+2^4}\Big)\eps^2 \Big)
	\leq T - T' \\
	\leq -\frac{1}{\kappa_1}\log\Big( 1-\frac{\eps^2}{4} - \Big(\frac{\eps}{\sqrt{2}b}\Big)^{\frac{4\kappa_1}{\kappa_2}}\Big).
	\end{multline*}
\end{lemma}

\begin{proof}
	The proof is essentially an application of Lemma \ref{lem:tri1}.
	We start with the upper bound. By \eqref{eq:lawcosin-kappa1} in Lemma \ref{lem:tri1}, the addition theorem for hyperbolic cosine,
	and the fact that $1-\cos(\eps)=2\sin^2(\eps/2)$ we estimate
	\begin{equation}\label{eq:cosh-sinh}
		\begin{split}
		\cosh(\kappa_1 T') &\geq \cosh(\kappa_1 T_1)\cosh(\kappa_1 T_2) -\sinh(\kappa_1 T_1) \sinh(\kappa_1 T_2)\cos(\pi-\eps)\\
		&= \cosh(\kappa_1(T_1+T_2)) +(\cos(\eps)-1)\sinh(\kappa_1 T_1)\sinh(\kappa_1 T_2)\\
		&= \cosh(\kappa_1(T_1+T_2)) -2\sin^2(\eps/2)\sinh(\kappa_1 T_1)\sinh(\kappa_1 T_2).
		\end{split}
	\end{equation}
	Moreover
	\begin{equation}\label{eq:frac-estimate}
		\begin{split}
		\frac{\sinh(\kappa_1 T_1)\sinh(\kappa_1 T_2)}{\cosh(\kappa_1 T)} &
		= \frac{1}{2}\frac{(e^{\kappa_1 T_1} -e^{-\kappa_1 T_1})(e^{\kappa_1 T_2} -e^{-\kappa_1 T_2})}{e^{\kappa_1 T} +e^{-\kappa_1 T}}\\
		&= \frac{(1-e^{-2\kappa_1 T_1})(1-e^{-2\kappa_1 T_2})}{2(1+e^{-2\kappa_1 T})} \leq \frac{1}{2}.
		\end{split}
	\end{equation}
	Using \eqref{eq:cosh-sinh} and \eqref{eq:frac-estimate} and since $e^x/2 \leq \cosh(x) \leq (e^x+1)/2$, we then estimate
	\begin{align} \label{eq:exp}
		\frac{e^{\kappa_1 T'} + 1}{e^{\kappa_1 T}} \geq \frac{\cosh(\kappa_1 T')}{\cosh(\kappa_1 T)} \geq 1 -\sin^2(\eps/2).
	\end{align}
	Using the assumption that $T_i$ satisfies \eqref{eq:period-loops1} for $i=1,2$ and again the fact that $1-\cos(\eps)=2\sin^2(\eps/2)$, we obtain
	\begin{equation}\label{eq:exp-T-bound}
	e^{\kappa_2 T_i} \geq \cosh(\kappa_2 T_i) \geq \frac{b}{1-\cos(\eps)} +1 = \frac{b+2\sin^2(\eps/2)}{2\sin^2(\eps/2)},
	\end{equation}
	so that
	\[
	e^{-\kappa_1 T}\leq \exp\Big( \frac{2\kappa_1}{\kappa_2} \log\big( \frac{2\sin^2(\eps/2)}{b+2\sin^2(\eps/2)}\big)\Big)
	= \Big( \frac{2\sin^2(\eps/2)}{b+2\sin^2(\eps/2)}\Big)^{\frac{2\kappa_1}{\kappa_2}}.
	\]
	Therefore from \eqref{eq:exp} we obtain
	\begin{align*}
		e^{\kappa_1(T'-T)} &\geq 1-\sin^2(\eps/2) -e^{-\kappa_1 T}\\
		& \geq 	1-\sin^2(\eps/2) -\Big( \frac{2\sin^2(\eps/2)}{b+2\sin^2(\eps/2)}\Big)^{\frac{2\kappa_1}{\kappa_2}}\\
		& \geq 1-\frac{\eps^2}{4} - \Big(\frac{\eps}{\sqrt{2}b}\Big)^{\frac{4\kappa_1}{\kappa_2}}
	\end{align*}
	and consequently
	\begin{align*}
		T'-T \geq \frac{1}{\kappa_1}\log\Big( 1-\frac{\eps^2}{4} - \Big(\frac{\eps}{\sqrt{2}b}\Big)^{\frac{4\kappa_1}{\kappa_2}}\Big),
	\end{align*}
	which proves the upper bound. 
	
	For the lower bound we proceed similarly using \eqref{eq:lawcosine-kappa2} in Lemma \ref{lem:tri1}:
	\begin{align*}
		\cosh(\kappa_2 T')&\leq \cosh(\kappa_2 T) +(\cos(\eps)-1)\sinh(\kappa_1 T_1)\sinh(\kappa_2 T_2) \\
		&= \cosh(\kappa_2 T) -2\sin^2(\eps/2)\sinh(\kappa_1 T_1)\sinh(\kappa_2 T_2).
	\end{align*}
	Since $T'\leq T$ we have
	\[
	\frac{\cosh(\kappa_2 T')}{\cosh(\kappa_2 T)} = \frac{e^{\kappa_2 T'} +e^{-\kappa_2 T'}}{e^{\kappa_2 T} +e^{-\kappa_2 T}}
	= e^{\kappa_2(T'-T)}\frac{1 +e^{-2\kappa_2 T'}}{1+e^{-2\kappa_2 T}} \geq e^{\kappa_2(T'-T)}.
	\]
	Moreover, using \eqref{eq:exp-T-bound} we estimate
	\begin{align*}
		\frac{\sinh(\kappa_2 T_1)\sinh(\kappa_2 T_2)}{\cosh(\kappa_2 T)} & = \frac{(1-e^{-2\kappa_2 T_1})(1-e^{-2\kappa_2 T_2})}{2(1+e^{-2\kappa_2 T})}\\
		& \geq \frac{\Big(1- \Big( \frac{2\sin^2(\eps/2)}{b+2\sin^2(\eps/2)}\Big)^2\Big)^2}{2\Big( 1+\Big( \frac{2\sin^2(\eps/2)}{b+2\sin^2(\eps/2)}\Big)^4\Big)}
		> \frac{1}{2}-\frac{4(b+2)^2}{(b+2)^4+2^4},
	\end{align*}
	since the function $\frac{(1-x^2)^2}{1+x^4}$ is decreasing for $x\in(0,1)$ and
	$\frac{2\sin^2(\eps/2)}{b+2\sin^2(\eps/2)}\in (0, \frac{2}{b+2})$ for $\eps\in(0,\pi)$.
	Therefore
	\begin{align*}
		e^{\kappa_2(T'-T)} &\leq \frac{\cosh(\kappa_2 T')}{\cosh(\kappa_2 T)}
		\leq 1-2\sin^2(\eps/2)\Big( \frac{1}{2}-\frac{4(b+2)^2}{(b+2)^4+2^4} \Big) \\
		& \leq 1-\frac{2}{\pi^2}\Big(\frac{1}{2}-\frac{4(b+2)^2}{(b+2)^4+2^4}\Big)\eps^2,
	\end{align*}
	since $\sin(x)\geq \frac{2}{\pi}x$ for $x\in(0,\pi/2)$, and consequently
	\[
	T'-T \leq \frac{1}{\kappa_2}\log\Big(  1-\frac{1}{\pi^2}\Big(1-\frac{8(b+2)^2}{(b+2)^4+2^4}\Big)\eps^2 \Big).
	\]
\end{proof}

We are now ready to prove Theorem 3 stated in the introduction.

\begin{theorem}\label{thm:partner-orbits-2dim}
	Let $M$ be a $2$-dimensional compact Riemannian manifold with sectional curvature $-\kappa_2^2 \leq K \leq -\kappa_1^2$ for $0<\kappa_1 \leq \kappa_2$.
	Let $\eps_0:= \frac{\pi}{16}\frac{\kappa_1}{\kappa_1+\kappa_2}$ and let $T_0=t_0(\kappa_1,\kappa_2)>1$ as in Proposition \ref{prop:c}.
	Then for all $\eps\leq\eps_0$, all $T_1, T_2\geq T_0$
	and all closed geodesics $c_w\colon[0,T]\rightarrow M$ of period $T=T_1+T_2$ with a self-crossing at $p=c_w(0)$ at time $T_1$ and crossing angle $\eps\leq\eps_0$
		there exist constants $C_1, C_2, C_3 >0$, depending on the curvature bounds and on the injectivity radius of $M$, and a closed geodesic $c_u\colon[0,T']\rightarrow M$ of period $T'<T$ such that
	\[
	C_1 \eps^2 \leq T-T'  \leq \left\{
	\begin{array}{ll}
		C_2 \eps^{\frac{4\kappa_1}{\kappa_2}} & \frac{\kappa_1}{\kappa_2}\in\big(0,\frac{1}{2}\big),\\
		C_3 \eps^2 &  \frac{\kappa_1}{\kappa_2}\in\big[\frac{1}{2},1\big],
	\end{array}\right.
	\]
	
\end{theorem}

\begin{proof}
	Let $c_w$ be as in the statement and denote by $c_{w,1}\colon[0,T_1]\rightarrow M$,
	$c_{w,1}(t)=c_w(t)$ and $c_{w,2}\colon[0,T_2]\rightarrow M$, $c_{w,2}(t)=c_w(T_1+t)$,
	where $T_2:=T-T_1$, its geodesic loops. Let also $\bar c_{w,2}$ such that $\bar c_{w,2}(t)= c_{w}(T-t)$,
	and $\gamma_1$, $\gamma_2$ be the deck transformation associated to $\tilde c_1=\tilde c_{w,1}$, $\tilde c_{w,2}$,
	respectively, which denote the lifts of $c_{w,1}$ and $c_{w,2}$.
	Following the beginning of the proof of Theorem \ref{thm:partner-general} consider the geodesic $\tilde c_2(0)=\gamma_2^{-1}\tilde \alpha(t)$
	and set $p_0:=\tilde c_1(0)$, $v_0:= \dot{\tilde c}_1(0)$,
	and $p_2:= \tilde c_2(T_2)$ so that $\gamma_1(p_0)=p_1$, $\gamma_2^{-1}(p_1)=p_2$,
	and $\sphericalangle_{p_1}( \dot{\tilde c}_1(T_1), \dot{\tilde c}_2(0) ) = \eps$.
	Moreover,
	\[
	\sphericalangle_{p_2}((\gamma_2^{-1}\gamma_1)_\ast v_0, \dot{\tilde c}_2(T_2))
	= \sphericalangle_{\pi w}( w, -\dot{c}_{w,2}(0)) = \eps.
	\]
	
	Denote by $\tilde c_3\colon[0,\hat T]\rightarrow M$ the geodesic connecting $p_0$ and $p_2$ where $\hat T=d(p_0,p_2)$.
	Let also $\theta_1:=\sphericalangle_{p_0}( v_0, \dot{\tilde c}_3(0))$ and $\theta_2:=\sphericalangle_{p_2}( \dot{\tilde c}_3(T'), \dot{\tilde c}_2'(T_2))$.
	
	Considering again the vector $v:=\dot{\tilde c}_3(0)$ and arguing as in the proof of Theorem \ref{thm:partner-general}, 
%
%
	we obtain that $\gamma_2^{-1}\gamma_1 \in \Gamma_{2\eps}(v_0,\hat T)$ and Proposition \ref{lemma:length-gamma}
	 yields the existence of an axis $c$ of $\gamma_2^{-1}\gamma_1$ such that $c(-\infty)\in \Pa_{\rho(2\eps)}$ and $\tilde c(+\infty)\in \Fu_{\rho(2\eps)}$,
	 for $\rho(2\eps)=4\Big(\frac{\kappa_2}{\kappa_1}+1\Big)\eps$.
%
	
	To estimate the action difference $\vert T- T'\vert$ where $T'=\vert \gamma_2^{-1}\gamma_1 \vert$  we proceed as follow.
	Parametrising $c$ such that $d(p_0, c)= d(p_0, c(0))$, we have
	\begin{align*}\label{eq:auxiliary}
		d(p_2, c(T')) & = d(\gamma_2^{-1}\gamma_1(p_0), \gamma_2^{-1}\gamma_1( c(0)) ) \\
		& = d(p_0, c(0))\leq a(\rho(2\eps))\leq \frac{8}{\kappa_1}\big(\frac{\kappa_2}{\kappa_1}+1\big)\eps,
	\end{align*}
	by Lemma \ref{lem:dis}.
	Further we have $\frac{\pi}{2}= \sphericalangle_{c(0)}(p_0, c(T')) = \sphericalangle_{c(T')}( p_2, c(2T'))$ which in turn gives
	$\sphericalangle_{c(T')}(p_2, c(0))=\frac{\pi}{2}$, due to dimension $2$. This means that $c[0,T']$ is the orthogonal projection of $\tilde c_3[0,\hat T]$ onto $c$.
	Hence, Lemma \ref{lem:perp-proj} (applicable because of dimension 2) with $r_1=r_2=d(p_0, c(0))$ yields
	\[
	\sinh\big(\kappa_2 \frac{\hat T}{2}\big)\leq \cosh(\kappa_2 d(p_0,c(0))\sinh\big( \kappa_2 \frac{T'}{2}\big),
	\]
	and since $\hat T\geq T'=\vert \gamma_2^{-1}\gamma_1\vert$ and using the estimate for $d(p_0, c(0))$ we further obtain
	\begin{align*}
		e^{\kappa_2\frac{\hat T-T'}{2}} \leq & \frac{e^{\kappa_2\hat T/2}(1-e^{-\kappa_2 \hat T})}{e^{\kappa_2 T'/2}(1- e^{-\kappa_2 T'})}
		= \frac{\sinh(\kappa_2\hat T/2)}{\sinh(\kappa_2 T'/2 )} \\
		& \leq \cosh\big(\kappa_2a(\rho(2\eps))\big).
	\end{align*}
	Using the definition of the function $a(\cdot)$, we conclude that there exists a constant $D>0$ depending on $\kappa_1$ and $\kappa_2$ such that
	\[
	e^{\kappa_2\frac{\hat T-T'}{2}} \leq 1+ D \eps^2,
	\]
	leading to
	\begin{equation}\label{eq:hat T - T'}
	\hat T- T'  \leq \frac{2}{\kappa_2}\log(1+ D\eps^2).
	\end{equation}
	Combining this estimate with the upper bound in Lemma \ref{lem:T'-bound} and using the Taylor series of $\log(1+x)$
	and of $\log(1-x)$ we find constants $C_2, C_3>0$ depending on $\kappa_1$, $\kappa_2$ and $\inj(M)$ such that
	\begin{align*}
		T- T' &= (T-\hat T) +(\hat T - T')\\
		&  \leq \frac{2}{\kappa_2}\log(1+ D\eps^2)
		-\frac{1}{\kappa_1}\log\Big( 1-\frac{\eps^2}{4} - \Big(\frac{\eps}{\sqrt{2}b}\Big)^{\frac{4\kappa_1}{\kappa_2}}\Big)\\
		& \leq \left\{
		\begin{array}{ll}
			C_2 \eps^{\frac{4\kappa_1}{\kappa_2}} & \frac{\kappa_1}{\kappa_2}\in\big(0,\frac{1}{2}\big)\\
			C_3 \eps^2 &  \frac{\kappa_1}{\kappa_2}\in\big[\frac{1}{2},1\big].
		\end{array}
	\right.
	\end{align*}
	Finally, since $\hat T- T' \geq 0$ the lower bound in Lemma \ref{lem:T'-bound} yields
	\begin{align*}
		T- T' & = (T-\hat T) +(\hat T - T') \\
		& \geq -\frac{1}{\kappa_2}\log\Big( 1-\frac{2}{\pi^2}\Big(1-\frac{4(b+2)^2}{(b+2)^4+2^4}\Big)\eps^2 \Big) \\
		& \geq C_1 \eps^2,
	\end{align*}
	for $C_1>0$ being the first non-zero coefficient of the Taylor expansion of the function on the second line at $\eps=0$.
	Combining the estimates above, we obtain the desired action difference.
	\end{proof}

\section{Construction of pseudo-partner orbits}\label{sec:pseudo-partner-orbits}

We now turn to the study of closed geodesics $c_w: [0, T] \to M$, where $M$ is again a manifold of arbitrary dimension, of period $T$
  having a large crossing angle at $c_w(0) = c_w(T_1)$ for $T_1 \in (0,T)$. We will express this fact assuming that
its complement angle $\sphericalangle_{\pi w}(w, \dot{c}_w(T_1))=\eps$ is small enough.
Indeed, if the latter angle is small,
the crossing angle $\sphericalangle_{\pi w}(w, -\dot{c}_w(T_1))$ is necessarily between $\pi-\eps$ and $\pi$, and therefore large.
Under this assumptions we show that a pair of closed geodesics exists which stay
close to the two loops generating the original closed geodesic and whose length are comparable but slightly smaller than the length of the loops. This pair of geodesics is often referred to as pseudo-partner orbits.



\begin{theorem}\label{thm:pseudo-partner}
Let $M$ be a compact Riemannian manifold with sectional curvature $-\kappa_2^2 \leq K \leq -\kappa_1^2$ for $0<\kappa_1 \leq \kappa_2$.
	Let $\eps_0:= \frac{\pi}{16}\frac{\kappa_1}{\kappa_1+\kappa_2}$ and let $T_0=t_0(\kappa_1,\kappa_2)>1$ be the one from Prop. \ref{prop:c}.
	Then for all closed geodesic $c_w: [0, T] \to M$ of period $T$
	with a self-crossing at $p=c_w(0)$ at time $T_1$ such that $\sphericalangle_{p}(w, \dot c_{w}(T_1))=\eps\leq \eps_0 $
	and $T_1, T_2:=T-T_1\geq T_0$
	there exists a pair of closed geodesics $c_{u_1}$ and $c_{u_2}$ of period $\hat T_1$ and $\hat T_2$,
	satisfying
	\[
	0\leq T_1-\hat T_1\leq \frac{8}{\kappa_1}\left(\frac{\kappa_2}{\kappa_1}+1\right)\eps
	\;  \;\text{and} \; \;
	0\leq T_2-\hat T_2\leq \frac{8}{\kappa_1}\left(\frac{\kappa_2}{\kappa_1}+1\right)\eps.
	\]
	Furthermore,
	\[
	d(c_{u_1}(s), c_w(s))\leq \frac{12}{\kappa_1}\Big(\frac{\kappa_2}{\kappa_1}+1\Big) \eps
	\]
	for $s\in[0,T_1]$, and
	\[
	d(c_{u_2}(s), c_w(s+T_1))\leq \frac{12}{\kappa_1}\Big(\frac{\kappa_2}{\kappa_1}+1\Big)\eps
	\]
	for $s\in[0,T_2]$.
	In particular,
	\[
	d(c_{u_1}(\hat T_1), c_{u_2}(0))\leq \frac{32}{\kappa_1}\left(\frac{\kappa_2}{\kappa_1}+1\right)\eps.
	\]

\end{theorem}


\begin{proof}
	 Let $c_w:[0,T]  \to M  $ be as in the statement of the theorem
	and denote by $c_{w,1}\colon[0,T_1]\to M$, $c_{w,1}(t)=c_w(t)$
	and $c_{w,2}\colon[0,T_2]\rightarrow M$, $c_{w,2}(t)=c_w(T_1+t)$, where $T_2:=T-T_1$ the associated geodesic loops.	
	Let $\tilde c_{w,1}$ be a lifts of $c_{w,1}$ to the universal cover $X$ and set $\tilde c_{w,1}(0)=:p_0\in X$, $\dot{\tilde c}_{w,1}(0):= v_0$, and
	$ \tilde c_{w,1}(T_1)=:p_1$.
	Let $\gamma_1$ be the deck transformation with  $\gamma_1(p_0) = p_1$. Then
	\[
	\sphericalangle_{p_1}((\gamma_1)_* v_0, \dot{\tilde c}_{w,1}(T_1)) = \sphericalangle_{\pi w}(w, \dot{c}_{w,1}(T_1))) =  \eps
	\]
	by assumption, which implies that $\phi^{T_1}(v_0)\in (\gamma_1)_* A_{\eps}(v_0)$.
	
	Therefore $v_0\in A_{\eps}(v_0)\cap \phi^{-T_1}(\gamma_1)_* A_{\eps}(v_0)$ and by Proposition \ref{lemm:length-gamma-2} and Remark \ref{rmk:improved-length} there exists an axis $c_1$ of $\gamma_1$ of length $\hat{T}_1$ such that	
	\begin{equation}\label{eq:length-first-axis}
		0 \le T_1- \hat{T}_1 \leq \frac{8}{\kappa_1}\Big(\frac{\kappa_2}{\kappa_1}+1\Big)\eps,
	\end{equation}
	and such that
	\begin{equation}\label{eq:dist-first-axis}
		d(c_1(s), \tilde c_{w,1}(s))\leq \frac{12}{\kappa_1}\Big(\frac{\kappa_2}{\kappa_1}+1\Big)\eps
	\end{equation}
	for all $s\in[0,T_1]$.
	
	A similar construction using a lift of the geodesic loop $c_{w,2}$ and the associated deck transformation $\gamma_2$ we obtain the existence
	of an axis $c_2$ of $\gamma_2$ of length $\hat{T}_2$ satisfying
	\begin{equation}\label{eq:length-second-axis}
		0 \le T_2- \hat{T}_2 \leq \frac{8}{\kappa_1}\Big(\frac{\kappa_2}{\kappa_1}+1\Big)\eps,
	\end{equation}
	and such that
	\begin{equation}\label{eq:dist-second-axis}
		d(c_2(s), \tilde c_{w,2}(s))\leq \frac{12}{\kappa_1}\Big(\frac{\kappa_2}{\kappa_1}+1\Big)\eps
	\end{equation}
	for all $s\in[0,T_2]$.
	
	Projecting the two axis $c_1$ and $c_2$ onto $M$ we obtain two closed geodesics $c_{u_1}=\pr \circ c_1(s)$ and $c_{u_2}=\pr \circ c_2(s)$
	of period $\hat T_1$ and $\hat{T}_2$, respectively, whose periods satisfy \eqref{eq:length-first-axis} and \eqref{eq:length-second-axis},
	and such that
	\[
	d(c_{u_1}(s), c_{w}(s))\leq \frac{12}{\kappa_1}\Big(\frac{\kappa_2}{\kappa_1}+1\Big)\eps,\qquad s\in[0,T_1],
	\]
	and
	\[
	d(c_{u_2}(s), c_{w}(s+T_1))\leq \frac{12}{\kappa_1}\Big(\frac{\kappa_2}{\kappa_1}+1\Big)\eps,\qquad s\in[0,T_2].
	\]
	Moreover, using these estimates and the triangle inequality we obtain
	\begin{align*}
	d(c_{u_1}(\hat T_1), c_{u_2}(0))& \leq d( c_{u_1}(\hat T_1), c_{w,1}(\hat T_1)) + d(c_{w,1}(\hat T_1), c_{w,1}(T_1)) \\
	&\hspace{1cm} + d(c_{w,2}(0), c_{u_2}(0))\\
	& \leq \frac{12}{\kappa_1}\Big(\frac{\kappa_2}{\kappa_1}+1\Big)\eps + (T_1-\hat T_1) + \frac{12}{\kappa_1}\Big(\frac{\kappa_2}{\kappa_1}+1\Big)\eps\\
	& \leq \frac{24}{\kappa_1}\Big(\frac{\kappa_2}{\kappa_1}+1\Big)\eps + \frac{8}{\kappa_1}\Big(\frac{\kappa_2}{\kappa_1}+1\Big)\eps
	= \frac{32}{\kappa_1}\Big(\frac{\kappa_2}{\kappa_1}+1\Big)\eps.
	\end{align*}
	which yields the last estimate.	
\end{proof}

\bibliographystyle{plain}

\bibliography{partner}

\end{document}